\documentclass[reqno,onepage]{amsart}

\usepackage[left=1in,right=1in,top=1in,bottom=1in]{geometry}

\usepackage{amsfonts}
\usepackage{amssymb, tikz}
\usepackage{mathtools}
\usepackage{amsthm}
\usepackage{amsaddr}
\usepackage{color}
\usepackage{caption}
\usepackage{subcaption}
\usepackage{graphicx}
\usepackage{hyperref}

\newtheorem{thm}{Theorem}
\newtheorem{cor}[thm]{Corollary}
\newtheorem{lem}[thm]{Lemma}
\newtheorem{prop}[thm]{Proposition}

\newtheorem{defn}[thm]{Definition}
\newtheorem{rem}[thm]{Remark}

\newcommand{\interior}{\mathrm{int}}

\newcommand{\blie}[1]{\mathrm{BLIE}_{#1}}
\newcommand{\flie}[1]{\mathrm{FLIE}_{#1}}
\newcommand{\cD}{\mathcal{D}}
\newcommand{\bS}{\mathbb{S}^{n-1}}
\newcommand{\ghd}[1]{d_{GH}^{#1}}
\newcommand{\hd}[1]{d_{H}^{#1}}
\newcommand{\N}{\mathbb{N}}
\newcommand{\R}{\mathbb{R}}
\newcommand{\cR}{\mathcal{R}}
\newcommand{\Z}{\mathbb{Z}}
\newcommand{\eps}{\varepsilon}

\newcommand{\CR}[1]{\cR_{X,S}}
\newcommand{\CRR}[1]{\cR_{#1}}
\newcommand{\p}{\partial}
\DeclareMathOperator{\diam}{diam}

\title[Lipschitz Stability of Travel Time Data]{Lipschitz Stability of Travel Time Data}
\date{\today}

\author{Joonas Ilmavirta}
\address{Department of Mathematics and Statistics,University of Jyv\"askyl\"a, Jyv\"askyl\"a,  Finland} 

\author{Antti Kykk\"anen}
\address{Department of Computational Applied Mathematics and Operations Research, Rice University, Houston, TX, USA}

\author{Matti Lassas}
\address{Department of Mathematics and Statistics, University of Helsinki, Finland}

\author{Teemu Saksala}
\address{Department of Mathematics, North Carolina State University, Raleigh, NC, USA   (\tt{tssaksal@ncsu.edu})}

\author{Andrew Shedlock}
\address{Department of Mathematics, North Carolina State University, Raleigh, NC, USA}

\begin{document}



\begin{abstract}
We prove that the reconstruction of a certain type of length spaces from their travel time data on a closed subset is Lipschitz stable. The travel time data is the set of distance functions from the entire space, measured on the chosen closed subset. The case of a Riemannian manifold with boundary with the boundary as the measurement set appears is a classical geometric inverse problem arising from Gel'fand's inverse boundary spectral problem. Examples of spaces satisfying our assumptions include some non-simple Riemannian manifolds, Euclidean domains with non-trivial topology, and metric trees.
\end{abstract}

\maketitle



\maketitle

\section{Introduction}

\textit{The travel time map} of a compact length space $X$ with a closed measurement set $S$ takes any point of the space to the continuous function which measures the distance from this point to every point in the measurement set. The range of this map is called the \textit{travel time data}. In this paper we study the choices of $X$ and $S$ for which the travel time map is a topological embedding and a $\eps$-local isometry, for some $\eps>0$ (see Definition \ref{defn:eps_isometry}), to the space of continuous functions on the measurement set $S$. Under these assumptions we show the travel time data determines the metric space Lipschitz stably. 

When the length space is a smooth compact Riemmannian manifold, and the measurement set is the boundary of the manifold, we derive a partial characterization of those manifolds whose travel time map satisfy the necessary $\eps$-local isometry property. On a Riemannian manifold with strictly convex boundary a sufficient condition would be that no geodesic is longer that twice the injectivity radius, while a necessary condition for our results is that no geodesic is longer than twice the diameter of the manifold. In the appendix of this paper we provide some concrete examples of Riemannian manifolds which fall under the scope of our main results. 

In our first main result Theorem \ref{thm:stability_2} we show that if the travel time maps of two compact length spaces, with a common closed measurement set, are both  topological embeddings, the inverses of the travel time maps are $\eps$-local isometries, and  the diameters of these spaces do not exceed some number $D>0$ then the Gromov--Hausdorff distance of these spaces can be estimated from above by a uniform constant times the Hausdorff distance of their travel time data. This constant depends only on the ratio of the numbers $D$ and $\eps$. In order to drop the diameter assumption of Theorem \ref{thm:stability_2} we introduce in Definition \ref{defn:truncated_GH_dist} the \textit{$\eps$-truncated Gromov--Hausdorff distance} which compares the similarity of  compact metric spaces only up to scale $\eps$.  In our second main result Theorem \ref{thm:stability_1} we show that if the inverses of the travel time maps of two compact length spaces with a common closed measurement sets are both $\eps$-local isometries, then the $\eps$-truncated Gromov--Hausdorff distance of these spaces is at most the Hausdorff distance of their travel time data. Furthermore, we show in Corollary \ref{cor:stability_2} that  the travel time map determines a smooth compact Finsler or Riemannian manifold with boundary up to a smooth isometry if the inverse of the travel time map of this manifold is $\eps$-local isometry for some $\eps>0$. For this result we prove a Myers--Steenrod theorem (Proposition \ref{prop:MS_thm}) for manifolds with boundary.


\subsection{Travel Time Data and the Main Results}

Every path-connected metric space $(X,d)$ has an induced \textit{length metric} $d_L$.
In this metric the distance $d_L(x,y)$ between $x,y \in X$ is the (possibly infinite) infimum over the lengths
\[
L(\gamma)=\sup \left\{ 
\sum_{i=0}^N d(\gamma(t_{i-1}),\gamma(t_i)):\:
0=t_0<t_1<t_2<\ldots <t_{N-1}<t_N=1
\right\},
\]
of all curves $\gamma\colon [0,1] \to X$ starting from $\gamma(0)=x$ and ending at $y=\gamma(1)$.
Clearly $d\leq d_L$.
The metric $d_L$ is symmetric, positive, and satisfies the triangle inequality (cf. \cite[Excersice 2.1.2. and Definition 2.3.1.]{burago2001course}), but it need not be finite.
The metric $d$ is called \textit{intrinsic} if $d=d_L$ and in this case $(X,d)$ is called a \textit{length space}.
The length metric induced by $d_{L}$ is simply $d_L$ itself.
Riemannian manifolds are length spaces by definition.

In this paper we study the stability of recovering a space from its \emph{travel time data}.

\begin{defn}
\label{defn:TTD}
Let $(X,d)$ be a compact length space and $S\subset X$ a closed set. For every point $p\in X$ its \emph{travel time function} $r_p:S \to \R$ is defined by the formula $r_p(z) = d(p, z)$. \emph{The travel time map} of the length space $(X,d)$ is then given by the formula
\begin{equation}
\label{eq:map_R}
\CR{} \colon(X,d) \to (C(S),\|\cdot\|_\infty), \quad \CR{} (p) = r_p.
\end{equation}
The image set 
$
\CR{} (X)\subset C(S)
$
of the travel time map is called the travel time data of the length space $(X,d)$. 
\end{defn}
For an inverse problem point of view our aim is to stably recover the compact length space $(X,d)$ from the travel time data of some ``fixed'' closed measurement set $S \subset X$. In particular, the travel time data is an unlabelled collection of travel time functions. The locations of point sources $p \in X$ related to the functions $r_p$ are unknown. It follows from the triangular inequality that the travel time map is always $1$-Lipschitz and due to compactness of $X$ the travel time data 
$
\CR{} (X)\subset C(S)
$
is  a compact subset of the Banach space $(C(S), \|\cdot\|_{\infty})$.

The travel time map is known to be a topological embedding if $X$ is a smooth manifold, $d$ is given by a smooth Riemannian metric, and $S$ is the smooth boundary of $X$ \cite[Lemma 3.30]{Katchalov2001}. Recently, this result was extended to reversible Finsler manifolds \cite{dehoop2020determination} by the first, the third and the fourth author and on simple gas giant metrics \cite{de2024geometric} by the first two authors.
A Finsler metric $F$ on a manifold $X$ is said to be reversible if $F(v)=F(-v)$ for all $v \in TX$, where $TX$ is the tangent bundle of $X$. Reversible Finsler manifolds have a symmetric distance function and they are length spaces. Geometrically, gas giant metrics correspond to a compact Riemannian manifold with boundary whose metric tensor has a conformal blow-up near the boundary.

On the other hand, if $X=\mathbb{S}^2$ and $S$ is the equator then the respective travel time map is not one-to-one since the travel time functions of the north and south poles agree. Thus ``small'' measurement sets can be problematic for the injectivity while sufficiently large ones are not. For instance, it was shown in \cite{helin2016correlation} by the third and the fourth author that if $X$ is any complete connected Riemannian manifold and $S$ is a closure of an open subset of $X$ then the respective travel time map is again a topological embedding. In Proposition \ref{thm:local_isometry} of the current paper we show that the same is true if $X$ is a Busemann G-space. 

Busemann G-spaces are metric spaces whose distance function satisfies a convexity condition. In particular, compact Busemann G-spaces are always length spaces whose distance minimizing curves do not split. Furthermore, on a Busemann space a shortest curve between any two points, that are close enough to each other, is always extendable as distance minimizing curves. The Busemann conjecture, which is a special case of the famous Bing--Borsuk conjecture, states that every Busemann G-space is a topological manifold. This conjecture is true for dimensions 1 through 4 \cite{halverson2008bing}. 

In the following definition we explain how to compare the closeness of the travel time data given by two length spaces with a ``common''  measurement set.

\begin{defn}
\label{defn:Haus_of_travel_time_data}
Let $(X_1,d_1)$ and $(X_2,d_2)$ be two compact length spaces with closed measurement sets $S_1 \subset X_1$ and $S_2 \subset X_2$.  If there is a homeomorphism $\phi\colon S_1 \to S_2$, we define 
\[
\CRR{X_2,S_2,\phi}(X_2):=\{d_2(q,\phi(\cdot))\colon S_1 \to \R | \: q \in X_2\}\subset C(S_1).
\] 

We say that \textit{the $\phi$-distance of travel time data} of the length spaces $(X_1,d_1)$ and $(X_2,d_2)$ is the number
\[
\hd{C(S_1)}(\CRR{X_1,S_1}(X_1),\CRR{X_2,S_2,\phi}(X_2)),
\]
where $\hd{C(S_1)}$ is the Hausdorff distance of $C(S_1)$.

Moreover, the travel time data of length spaces $(X_1,d_1)$ and $(X_2,d_2)$  \textit{coincide} if 
\[
\hd{C(S_1)}(\CRR{X_1,S_1}(X_1),\CRR{X_2,S_2,\phi}(X_2))=0,
\]
for some homeomorphism $\phi\colon S_1\to S_2$.
\end{defn}

In this paper we show that if the travel time data of two compact length spaces, with a common measurement set, are close to each other then the metric spaces need to be close to each other. To quantify the closeness of compact metric spaces $X$ and $Y$ we rely on the \emph{Gromov--Hausdorff distance:}
\[
\begin{split}
\ghd{}(X,Y):=
\inf
\{
\hd{Z}(f(X),g(Y));
&
\: Z\text{ is a metric space, }
f\colon X\to Z
\text{ and }
\\& \: g\colon Y\to Z
\text{ are isometric embeddings}
\}.
\end{split}
\]




To specify the class of compact length spaces studied in this paper we introduce the following local concept of a metric isometry.

\begin{defn}
\label{defn:eps_isometry}
    Let $\eps>0$, also let  $(X,d_X)$ and $(Y,d_Y)$ be metric spaces. We say that a continuous map $\phi \colon X \to Y$ is an $\eps$-local isometry if it satisfies the following property: If $p,q \in X$ are such that $d_X(p,q)<\eps$ then
    \[
    d_X(p,q)=d_Y(\phi(p),\phi(q)).
    \]
\end{defn}

We are ready to define the class of metric spaces we are going to work with.

\begin{defn}
\label{defn:special}
    Let $(X,d)$ be a compact length space with a closed measurement set $S$ such that $\CR{}\colon (X,d) \to (C(S), \|\cdot\|_{\infty})$
    is a topological embedding and let $\eps > 0$. We say that $(X,d)$ with measurement set $S$ is 
\begin{itemize}
        \item[(I)] Forward $\eps$-Locally Isometrically Embeddable ($\flie{\eps}$) if $\CR{}\colon (X,d) \to (C(S), \|\cdot\|_{\infty})$ is an $\eps$-local isometry.
        
        \item[(II)] Backward $\eps$-Locally Isometrically Embeddable ($\blie{\eps}$) if $\CR{}^{-1}\colon (\CR{}(X),\|\cdot\|_{\infty}) \to (X,d)$ is an $\eps$-local isometry.
    \end{itemize}
\end{defn}

Clearly the $\flie{\eps}$ and $\blie{\eps}$-spaces are nested with respect to the $\eps$-parameter. We note that it is possible for a bijective map $f:X\to Y$ to be an $\eps$-local isometry, while the inverse map $f^{-1}\colon Y\to X$ is not. For example, consider the function $f\colon [0,\frac{3\pi}{2}] \to \mathbb{S}^1$, where $\mathbb{S}^1$ has the round metric, and $f(t) = (\cos(t),\sin(t))$. Then $f$ is a  $\pi$-local isometry but $f^{-1}$ is only a $\frac{\pi}{2}$-local isometry. Thus, $\flie{\eps}$ and $\blie{\eps}$ conditions are not necessary equivalent, but they are strongly connected to each other since the travel time map is always $1$-Lipschitz. 
In Proposition \ref{thm:delta_test} of this paper we show that if $X$ and $S$ satisfy the $\blie{\eps}$-property for some $\eps>0$ then they also satisfy the $\flie{\eps}$-property. Conversely, in Proposition \ref{prop:(eps,delta)_isometry} we use a compactness argument to show that if $X$ and $S$ are $\flie{\eps}$ then they are also $\blie{\eps_0}$ for some possibly smaller $\eps_0 >0$.
%
%
Due to the remarks we made earlier about the injectivity of the travel time map we show in Proposition \ref{thm:local_isometry} that $\CR{}$ is a topological embedding if $X$ and $S$ are in the following two categories:
\begin{itemize}
    \item[(a)] $X$ is a compact manifold and $S \subset X$ is the boundary of $X$.
    
    \item[(b)] $X$ is a compact Busemann G-space and the measurement set $S\subset X$ is the closure of a non-empty open set.
\end{itemize}
A smooth compact Riemannian manifold $X$ with boundary $S$ is called \textit{simple}
if $S$  is strictly convex, $X$ is simply connected and no geodesic of $X$ has conjugate points. See for instance \cite[Preface]{paternain2023geometric}. It was shown recently in \cite[Proposition 6]{ilmavirta2023three} by the first and the fourth author that the travel time map $\CR{}$ of a simple Riemannian manifold is always an isometry. Thus, a simple Riemannian manifold satisfies the $\flie{\eps}$ and $\blie{\eps}$-properties for every $\eps>0$. 

In Section \ref{sec:scope} we study some geometric properties which yield the $\flie{\eps}$-property. For instance, if $(N,g)$ is a closed Riemannian manifold and $X\subset N$ is an open set with smooth strictly convex boundary $S$ then the triplet $(X,g,S)$ satisfied the  $\flie{\eps}$-condition for some $\eps>0$ if for any $N$-geodesic each connected component of this geodesic in $X$ is strictly shorter than twice the injectivity radius of the ambient manifold $(N,g)$. This observation extends the scope of this paper beyond simple manifolds since if $X \subset N$ was simple then each of its points has a normal neighborhood containing $X$. In the Appendix \ref{Sec:Appendix} of this paper we provide examples of manifolds with boundary which satisfy the $\flie{\eps}$-property for some $\eps>0$, but nor for all. In particular we give an example of a non-convex and non-simply connected $\flie{\eps}$-domains in $\R^n$, 
and provide a numerical evidence for the existence of a $\flie{\eps}$-manifold with strictly convex boundary and interior conjugate points.

In Proposition \ref{thm:local_isometry} we show that the travel time map of each compact Busemann G-space with a  measurement set, which is a closure of a non-empty open set, is always a topological embedding.  Due to the Busemann conjecture we focus on Riemannian manifolds in our set of examples. In Lemma \ref{lem:lenght_of_geos}, we show that a closed Riemannian manifold is not in the $\flie{\eps}$-class for any $\eps>0$ if the set $X \setminus S$ has a trapped geodesic. Clearly, the $\flie{\eps}$-condition is valid for all $\eps>0$ if the set $X \setminus S$ is contained in a simple Riemannian manifold. In Appendix \ref{Sec:Appendix} we show that $\mathbb{S}^2$ is $\flie{\eps}$ for some but not all $\eps>0$ if the measurement set $S$ contains a trace of a geodesic connecting two antipodal points. Thus, there are examples of $\flie{\eps}$-manifolds when $X\setminus S$ is not simple. 
 
\begin{rem}
In the case (b) we are allowed to measure the distances through the set $S$ while in case (a),  the boundary works as an ``obstacle'' and we do not measure any distances through it even if the manifold is embedded into some ambient space. 

Examples (a) and (b) are mutually exclusive since shortest curves on a manifold with  boundary can branch. For instance, consider a planar domain whose boundary has a strictly concave part.

The assumption of non-branching distance minimizing curves in (b) is convenient as under this assumption the travel time map $\CR{}$ on a space $X$ is always injective, if the measurement set $S$ is a closure of an open set. In particular the incjectivity of the map $\CR{}$ is not dependent on the location of $S$. 
\end{rem}

The first main theorem of the current paper is the following.

\begin{thm}[Stability of Travel Time Data: First Version]
\label{thm:stability_2}
Let $\eps, D > 0$. Let $(X_1,d_1)$ and $(X_2,d_2)$ be two compact length spaces with closed measurement sets $S_1 \subset X_1$ and $S_2 \subset X_2$.  
Suppose that the diameters of $X_1$ and $X_2$ are less than $D$, and both spaces are $\blie{\eps}$. 
If there is a homeomorphism $\phi\colon S_1 \to S_2$,
then
\begin{equation}
    \label{eq:dist_of_manifolds}
      \ghd{} ((X_1,d_1),(X_2,d_2))
    \leq\Big(\frac{2D}{\eps} + 1\Big) 
    \hd{C(S_1)}(\CRR{X_1,S_1}(X_1),\CRR{X_2,S_2,\phi} (X_2)).
\end{equation}
If the travel time data of spaces $X_1$ and $X_2$ coincide, then these metric spaces are isometric. 
\end{thm}


In order to drop the diameter assumption in Theorem \ref{thm:stability_2} we study the similarities of compact metric spaces up to a scale $\eps>0$. For this purpose we introduce the following two definitions. The first definition introduces a certain \textit{truncation} of a metric space.

\begin{defn}
\label{defn:truncation}
    Let $(X,d)$ be a metric space and $\eps>0$. We call the metric space $(X,d_\eps)$, where 
    \[
    d_\eps(x,y):=\min\{d(x,y), \eps\}, \quad \text{ for all } x,y \in X
    \]
    the $\eps$-truncation of $(X,d)$.
\end{defn}

It is straight forward to prove that if $(X,d)$ is a metric space then all of its $\eps$-truncations are also metric spaces whose diameter do not exceed $\eps$. Furthermore, the spaces $(X,d)$ and $(X,d_\eps)$ are homeomorphic. The next definition introduces the \emph{$\eps$-truncated Gromov--Hausdorff distance} which is used to compare the closeness of $\eps$-truncated metric spaces.

\begin{defn}
\label{defn:truncated_GH_dist}
    Let $\eps>0$, and let $(X,d_X)$ and $(Y,d_Y)$ be compact metric spaces. Then we call the number
    \[
    \ghd{\eps}(X,Y):=\ghd{}((X,d_{X,\eps}),(Y,d_{Y,\eps})).
    \]
    the $\eps$-truncated Gromov--Hausdorff distance.
\end{defn}

Our second main result is as follows.

\begin{thm}[Stability of Travel Time Data: Second Version]
\label{thm:stability_1}
Let $\eps > 0$. 
Let $(X_1,d_1)$ and $(X_2,d_2)$ be two compact length spaces with closed measurement sets $S_1 \subset X_1$ and $S_2 \subset X_2$. If these spaces are both $\blie{\eps}$ and if there is a homeomorphism $\phi\colon S_1 \to S_2$
then
\begin{equation}
    \label{eq:dist_of_manifolds_2}
    \ghd{\eps} ((X_1,S_1),(X_2,S_2))\leq \hd{C(S_1)}(\CRR{X_1,S_1}(X_1),\CRR{X_2,S_2,\phi}(X_2)).
\end{equation}
If the travel time data of spaces $X_1$ and $X_2$ coincide, then these metric spaces are isometric.
\end{thm}

Theorem~\ref{thm:stability_1} yields the following corollary.

\begin{cor}
\label{cor:stability_2}
Let $(M_1,g_1)$ and $(M_2,g_2)$ be two compact reversible Finsler manifolds. Suppose that
\begin{itemize}
    \item[(a)] $S_i\subset M_i$ is the smooth boundary of $M_i$ for both $i \in \{1,2\}$ or
    \item[(b)] $M_i$ has no boundary and $S_i\subset M_i$ is a closure of a non-empty open set for both $i \in \{1,2\}$.
\end{itemize}
If for $i \in \{1,2\}$ the Finsler manifolds $(M_i,g_i)$ with the closed measurement sets $S_i$ are both in the $\blie{\eps}$-class for some $\eps>0$, and
if their travel time data coincide then the Finsler manifolds $(M_1,g_1)$ and $(M_2,g_2)$ are Finslerian isometric.
\end{cor}

We end this subsection by pointing out that theorems \ref{thm:stability_2} and \ref{thm:stability_1} can be seen as generalizations of \cite[Theorem 9]{ilmavirta2023three} in which a Lipschitz-stability was provided for the travel time data of simple Riemannian manifolds. Moreover, these results should be seen as quantitative versions of \cite{de2024geometric, de2021determination, helin2016correlation, Katchalov2001}. The result of Corollary \ref{cor:stability_2} is not new but its proof, which is based on the Meyers-Steenrod theorem: Every metric isometry between smooth Riemannian or Finsler manifolds is a smooth map which preserves the metric, has not been presented in the earlier literature in this generality. Our new proof streamlines the old uniqueness proofs as it does not require an independent reconstructions of topology, differentiable structure or the metric from the travel time data. Alas, it comes with the prize of the additional $\blie{\eps}$-assumption.
Up to the best knowledge of the authors theorems \ref{thm:stability_2} and \ref{thm:stability_1} are the first more generally applicable Lipschitz stability results for the travel time data. Thus, these results form an important stepping stone for any future computational study related this topic.

\begin{rem}
Metric trees\footnote{A tree is a connected graph without loops. The nodes of degree one are called leaves. A metric tree consists of all the edges as intervals (not necessarily unit length), glued together at the vertices.} with leaves as the measurement set have an isometric boundary distance embedding and thus satisfy our  $\flie{\eps}$ and $\blie{\eps}$ -conditions for all $\eps>0$.
Therefore,  theorems \ref{thm:stability_2} and \ref{thm:stability_1} hold for these graphs.
For controlling structural similarities of metric trees with the Gromov--Hausdorff distance, we refer to~\cite{Vandaele-graph}.
\end{rem}

\subsection{Some related inverse problems}
 The problem of determining the isometry type of a compact Riemannian manifold from its travel time data, as in Definition \ref{defn:TTD}, was introduced for the first time in~\cite{kurylev1997multidimensional}. The reconstructions of a smooth atlas on the manifold and the metric tensor was originally considered in~\cite{Katchalov2001}. In contrast to the paper at hand, the uniqueness result does not need any extra assumption for the geometry. If the travel time data is measured only on some open subset of the boundary then the unique recovery of a compact Riemannian manifold with strictly convex boundary is still possible \cite{pavlechko2022uniqueness}. 


The travel time data is related to many other geometric inverse problems. For instance, its recovery is a crucial step in proving uniqueness for Gel'fand's inverse boundary spectral problem~\cite{Katchalov2001}.
Gel'fand's problem concerns the question whether the boundary spectral data $(\p M, (\lambda_j, \p_\nu \phi_j|_{\p M})_{j=1}^\infty) $ determine $(M,g)$ up to isometry, when $(\lambda_j, \phi_j)$ are the Dirichlet eigenvalues and the corresponding $L^2$-orthonormal eigenfunctions of the Laplace--Beltrami operator.
Belishev and Kurylev provide an
affirmative answer to this problem in~\cite{belishev1992reconstruction} by developing the celebrated boundary control method. Stability for the boundary spectral data was developed in \cite{burago2020quantitative}. Stability for interior spectral data on a closed manifold was developed in \cite{bosi2022reconstruction}. Both of these papers prove a $\log$-$\log$-type stability results, while the optimal stability of this problem is conjectured to rather be of $\log$-type.

In \cite{katsuda2007stability} the authors studied a question of approximating a Riemannian manifold under the assumption: For a finite set of receivers $R\subset \p M$ one can measure the travel times $d(p,\cdot)|_{R}$ for finitely many $p \in P\subset  M^\interior$ under the \textit{a priori} assumption that $R \subset \p M$ is $\eps$-dense and that $\{d(p,\cdot)|_{R}: \: p \in P\}\subset \{d(p,\cdot)|_{R}: \: p \in M^\interior\}$ is also $\eps$-dense.  Thus $\{d(p,\cdot)|_{R}: \: p \in P\}$ is a finite measurement.  The authors construct an approximate finite metric space  $M_\eps$ and show that the Gromov--Hausdorff distance of $M$ and $M_\eps$ is proportional to some positive power of $\eps$. Thus, they provided a H\"older stability result for the travel time data. In \cite{katsuda2007stability} an independent travel time measurement is made for each interior source point in $P$, whereas in \cite{de2021stable} the authors studied the approximate reconstruction of a simple Riemannian manifold 
by measuring the arrival times of wave fronts produced by several point sources, that go off at unknown times, and moreover, the signals from the different point sources are mixed together.  To describe the similarity of two metric spaces `with the same boundary' the authors defined a labeled Gromov--Hausdorff distance. This is an extension of the classical Gromov--Hausdorff distance which compares both the similarity of the metric spaces and the sameness of the boundaries --- with a fixed model space for the boundary. In addition to reconstructing a discrete metric space approximation of $(M,g)$, the authors in \cite{de2021stable} estimated the density of the point sources and established an explicit error bound for the reconstruction in the labeled Gromov--Hausdorff sense.
Lipschitz stability of travel time data for simple Riemannian manifolds was recently proved in \cite{ilmavirta2023three}.


Taking the difference of the arrival times one obtains a \textit{boundary distance difference function}
$D_{p}(z_1,z_2):=d(p,z_1)-d(p,z_2)$
for all $z_1,z_2 \in \p M$.
In~\cite{lassas2015determination}
it was shown that if $U\subset N$ is an open subset of a closed
Riemannian manifold $(N,g)$ with a non-empty interior, then
\emph{distance difference data} 
$((U,g|_{U}), \{D_p\colon U\times U
\to \R \:| \:p \in N\})$
determine $(N,g)$ up to an isometry. This result was generalized for complete Riemannian manifolds~\cite{ivanov2018distance} and for compact Riemannian manifolds with boundary~\cite{de2018inverse, ivanov2020distance}. As part of \cite{ivanov2018distance}, the author established a stability without an explicit modulus of continuity. In \cite{ilmavirta2023three} the authors establish Lipscitz stability of the distance difference data for simple Riemannian manifolds. 

If the sign in the definition of the distance difference functions is
changed, we arrive in the distance sum functions,
$D^+_p(z_1,z_2)=d(z_1,p)+d(z_2,p)$
for all $p\in M$ and $z_1,z_2\in \p M$.
These functions give the lengths of the broken geodesics, that is, the
union of the shortest geodesics connecting~$z_1$ to~$p$ and the
shortest geodesics connecting~$p$ to~$z_2$.
Also, the gradients of $D^+_p(z_1,z_2)$ with respect to~$z_1$ and $z_2$ give the velocity
vectors of these geodesics.
The inverse problem of determining the
manifold $(M,g)$ from the \emph{broken geodesic data}, consisting of
the initial and the final points and directions, and the total length
of the broken geodesics, has been considered in~\cite{kurylev2010rigidity}. The proof is based on a reduction from the broken geodesic data to the travel time data. Unfortunately, the proof only works for compact smooth manifold of dimension three and higher, while under the simplicity assumption the dimensional restriction was removed in \cite{ilmavirta2023three}.
A different variant of broken geodesic data was recently considered in~\cite{meyerson2020stitching}.

The authors of~\cite{de2021determination} went beyond the conventional Riemannian setting and studied the recovery of a compact Finsler manifold from its travel time data. In contrast to earlier Riemannian results \cite{Katchalov2001, kurylev1997multidimensional} the data only determines the topological and smooth structures, but not the global geometry. However, the Finsler function $F\colon TM \to [0,\infty)$ can be recovered in a closure of the set $G(M,F)\subset TM$, which consists of points $(p,v)\in TM$ such that the corresponding geodesic~$\gamma_{p,v}$ is distance minimizing to the terminal boundary point. 
In~\cite{de2020foliated} the main result of~\cite{de2021determination} was utilized to generalize the result of~\cite{kurylev2010rigidity}, about the broken geodesic data, on reversible Finsler manifolds, satisfying a convex foliation condition.

\subsection{Organization of this paper}
In Proposition \ref{prop:(eps,delta)_isometry} of Section \ref{sec:scope} we show that each $\flie{\eps}$-space satisfies the $\blie{\eps}$-condition, after possible choosing a smaller $\eps$. The end of the section is reserved for further properties of $\flie{\eps}$-spaces.

In Section \ref{sec:length_Spaces} we survey some properties of length spaces and show in Proposition \ref{thm:isometry_of_length_spaces} that the travel time map of any $\flie{\eps}$-space preserves the length structure. In Lemma \ref{lem:props_of_trunc} we provide a list of many important properties between metrics and their $\eps$-truncations.

In Proposition \ref{thm:delta_test} of Section \ref{sec:verification} we provide a data driven method that can be used to find the largest $\eps>0$ such that a manifold which is \textit{a priori} known to be $\flie{e_0}$ for some $\eps_0>0$ is also $\blie{\eps}$. We also show in this same proposition that the $\blie{\eps}$-condition is stronger than the $\flie{\eps}$-condition. Thus, in our main theorems we do not need to assume any local isometry properties for the travel time map, only for its inverse.

In Section \ref{Sec:proofs} we provide the Myers--Steenrod theorem (Proposition \ref{prop:MS_thm}) for compact manifolds with boundary. Also, in Proposition \ref{prop:ghd_equivalent} we derive the Lipschitz constant appearing in the estimate \eqref{eq:dist_of_data}. Finally, we prove the main theorems \ref{thm:stability_2} and \ref{thm:stability_1}.

The Appendix is dedicated to examples of $\flie{\eps}$-manifolds.

\subsection*{Acknowledgements}

J.I. and A.K. were supported by the Research Council of Finland (Flagship of Advanced Mathematics for Sensing Imaging and Modelling grant 359208; Centre of Excellence of Inverse Modelling and Imaging grant 353092; and other grants 351665, 351656, 358047).
A.K. was also supported by the Finnish Academy of Science and
Letters.

M.L. was supported by PDE-Inverse project of the European Research Council of the European Union. M.L. was also supported by Academy of Finland grants 273979 and 284715. 

T.S. and A.S. were supported by the National Science Foundation (DMS-2204997). T.S. and A.S. would like to express their gratitude to Matti Lassas and Lauri Oksanen who kindly hosted them at University of Helsinki, Finland. Most of the scientific discoveries reported in this paper were made during that visit.

Views and opinions expressed are those of the authors only and do not necessarily reflect those of the European Union or the other funding organizations. Neither the European Union nor the other funding organizations can be held responsible for them. 

\section{The scope of this paper: When is a compact length space  is \texorpdfstring{$\flie{\eps}$}{BE}
or 
\texorpdfstring{$\blie{\eps}$}{BE}?}
\label{sec:scope}

\begin{prop}
\label{thm:local_isometry}

The travel time map $\CR{}$ of a compact length space $(X,d)$ with a closed measurement set $S$ is a topological embedding if
\begin{itemize}
    \item[(a)] $X$ is a compact manifold with boundary, whose distance function is given by (i) Riemannian, (ii) reversibly Finslerian or (iii) gas giant metric, and $S$ is the boundary of $X$

    \item[(b)] $(X,d)$ is a Busemann G-space (such as a closed Riemannian manifold) and $S$ is a closure of a non-empty open set of $X$.

\end{itemize}
Furthermore, $(X,d)$ and the measurement set $S \subset X$ satisfy the the $\flie{\eps}$-property if and only for every $p,q \in X$ such that $d(p,q) < \eps$  
there is a distance minimizing curve from $p$ to $q$ (or from $q$ to $p$) which extends as a distance minimizing curve to some point in $S$.


\end{prop}
\begin{proof}
The proofs for the first two claims of part (a) can be found in
\cite[Lemma 3.30]{Katchalov2001} for (i)
and
\cite[Proposition 3.1]{de2021determination} for (ii).

Part (iii) was proven in~\cite[Theorem 16]{de2024geometric} when the gas giant manifold is simple.
The proof in the general case is analogous to the Riemannian one in~\cite[Lemma 3.30]{Katchalov2001}.
We only have to establish that any geodesic from any point~$p$ to a boundary point~$z$ that minimizes the distance from~$p$ to~$\partial M$ meets the boundary conormally at~$z$.
On a gas giant every geodesic meets the boundary normally, so we have to be careful to phrase normality on the cotangent side.
Existence of a minimizing geodesic follows from standard metric theory.

Let us then show this conormality.
Without changing the minimizing geodesic
we may assume that the point~$p$ is arbitrarily close to its closest boundary point $z$.
This ensures that the distance to~$p$ is a smooth function in a neighborhood of~$z$ on~$M$ in the smooth structure corresponding to boundary normal coordinates.
The differential of this distance at any point~$q$ is the momentum $\xi_{pq}|_q\in T_q^*M$ of the unique geodesic connecting $p$ to~$q$.
For~$z$ to be a closest boundary point, the differential of the restriction of this distance to $\partial M$ must vanish, requiring $\xi_{pz}|_z\in T_z^*M$ to be conormal to the boundary as claimed.

For part (b) we suppose that $p,q \in X$ are such that $r_p=r_q$. That is 
\[
d(p,x)=d(q,x), \quad \text{for all } x \in S.
\]
Let $x_0 \in S$ be an interior point of $S$, $t_0:=d(p,x_0)$ and let $\gamma\colon [0,t_0] \to X$ be a distance minimizing curve from $x_0$ to $p$ parameterized by the arc length. Since $x_0$ is an interior point of $S$ there is $t\in (0,t_0)$ such that $\gamma([0,t])\subset S$. Hence, for $x:=\gamma(t)$ we have that
\[
d(q,x_0)=d(p,x_0)=d(x_0,x)+d(x,p)=d(x_0,x)+d(x,q).
\]
Thus, if $\alpha\colon[t,t_0] \to X$ is a distance minimizing curve from $x$ to $q$ then the concatenate curve $\tilde \gamma:= \alpha \circ \gamma|_{[0,t]}$ is a shortest curve from $x_0$ to $q$. Since $X$ is a Busemann G-space and the curve $\tilde \gamma$ agrees with the curve $\gamma$ on the interval $[0,t]$ it must hold that these curves are the same.
Thus,
$
p=\gamma(t_0)=q,
$
and we have proved that the travel time map $\CR{} $ is injective. Since this map is continuous and $X$ is compact, $\CR{} $ is also closed. We have proved the first claim for the case (b).

Finally we prove the characterization of the $\flie{\eps}$-property.
Let a closed set $S\subset X$ be such that the map $\CR{} $ is a $\eps$-local isometry. Let $p,q \in X$ be such that $d(p,q)<\eps$. Since the set $S$ is compact there is $z\in S$ such that
\[
|d(p,z)-d(q,z)|=d(p,q).
\]
Thus,
\[
d(p,z)=d(p,q)+d(q,z), \quad \text{or} \quad d(q,z)=d(p,q)+d(p,z).
\]
Since $(X,d)$ is a length space we get from these equations that there is a distance minimizing curve $\sigma$ from $p$ to $z$,  which goes through $q$, or there is distance minimizing curve $\tilde \sigma$ from $q$ to $z$,  which goes through $p$.


Let $x,y \in X$, be such that $d(x,y)<\eps$. Assume that there is a distance minimizing curve, $\sigma \colon [0,t_2]\to X$ such that $\sigma(t_2)=y$, $z:=\gamma(0)\in S$, and for some $t_1 \in [0,t_2]$ we have that $\sigma(t_1)=x$. Thus, by the reverse triangle inequality we have that
\[
d(x,y)=t_2-t_1=d(y,z)-d(x,z)=\sup_{q \in S}|d(y,q)-d(x,q)|=\|\CR{} (y)-\CR{} (x)\|_{\infty}.
\]
We have proved that $\CR{} $ is a $\eps$-local isometry.
\end{proof}

\subsection{Each \texorpdfstring{$\flie{\eps}$}{FLIE}-space is a \texorpdfstring{$\blie{\eps}$}{BLIE}-space after possible choosing a smaller \texorpdfstring{$\eps$}{epsilon}}


In the following proposition we show that if a compact length space with a closed measurement set is in the class $\flie{\eps_0}$ for some $\eps_0>0$ then there is $\eps\in (0,\eps_0]$ such that the the space is also in the $\blie{\eps}$-class. Further, it shows that for a space that is $\blie{\eps}$, then $\CR{}$ is a metric isometry between truncated spaces.

\begin{prop}
\label{prop:(eps,delta)_isometry}
If a compact length space $(X,d)$ with a closed measurement set $S$ is $\flie{\eps_0}$ for some $\eps_0>0$, then there is $\eps\in (0,\eps_0]$ such that the inverse map
\[
\CR{}^{-1}\colon (\CR{}(X),\|\cdot\|_\infty) \to (X,d)
\]
is a $\eps$-local isometry and the map
\[
\CR{}\colon (X,d_{\eps}) \to (C(S),d_{\infty,\eps})
\]
is a metric isometry, where $d_{\eps}$ and $d_{\infty,\eps}$ are the $\eps$-truncations of the metric $d$ and the supremum metric of $C(S)$ respectively.
\end{prop}
\begin{proof}
Due to our assumptions $\CR{}\colon (X,g) \to (C(S), \|\cdot\|_\infty)$ is a topological embedding, and since $X$ is compact the image set $\CR{}(X)\subset C(S)$ is compact.
We note that the collection 
\[
\Big\{\CR{}\Big(B\Big(\CR{}^{-1}(y),\frac{\eps_0}{2}\Big)\Big)\subset \CR{}(X): \: y \in \CR{}(X)\Big\}
\]
is an open cover of $\CR{}(X)$ with respect to the subspace topology of $\CR{}(X)$. Here, $B(x,r)$ is a metric ball of $(X,d)$ centered at $x \in X$ of radius $r>0$.  Thus, for every $y \in \CR{}(X)$ there is an open set $V_y \subset C(S)$ such that
\[
V_y\cap \CR{}(X)=\CR{}\Big(B\Big(\CR{}^{-1}(y),\frac{\eps_0}{2}\Big)\Big).
\]
Since $\CR{}(X)$ is a compact set with an open cover $\{V_y\}_{y\in \CR{}(X)}$ from the metric space $(C(\p M), \|\cdot\|_{\infty})$ it follows from the Lebesgue's number lemma that there is $\eps\in (0,\eps_0]$ such that for every $y \in \CR{}(X)$ there exists $z\in \CR{}(X)$ such that
\[
B_{\infty}\Big(y,\frac{\eps}{2}\Big)\subset V_{z}.
\]
In above, the set $B_\infty(y,\frac{\eps}{2})$ is a metric ball of $(C(S), \|\cdot\|_{\infty})$. By the definition of $\flie{\eps_0}$, the map $\CR{}$ is a $\eps_0$-local isometry. Therefore, the inverse map 
$\CR{}^{-1}\colon (\CR{}(X),\|\cdot\|_{\infty})\to (X,d)$ is an isometry on the set $\CR{}(B(\CR{}^{-1}(z),\frac{\eps_0}{2}))$, which contains the induced ball
$
B_\infty(y,\frac{\eps}{2}) \cap \CR{}(X).
$
This means that $\CR{}^{-1}\colon (\CR{}(X),\|\cdot\|_{\infty})\to (X,d)$ is a $\eps$-local isometry.

Finally, we choose any $p,q \in X$, and aim to show that $d_{\eps}(p,q)=d_{\infty,\eps}(\CR{}(p),\CR{}(q))$. First note that $d_\infty(\CR{}(p),\CR{}(q)) \leq d(p,q)$ since $\CR{}$ is always a $1$-Lipschitz function by the reverse triangle inequality. We also have that
\[
    d_{\infty,\eps}(\CR{}(p),\CR{}(q))\leq d_\eps(p,q) \leq \eps
\]
If $d_{\infty}(\CR{}(p),\CR{}(q)) \geq \eps$ then $d_{\infty,\eps}(\CR{}(p),\CR{}(q)) = d_\eps(p,q)$. If $d_\infty(\CR{}(p),\CR{}(q)) < \eps$, then since $\CR{}^{-1}$ is an $\eps$-local isometry we have that $d_\infty(\CR{}(p),\CR{}(q)) = d(p,q) < \eps$ and we find that $d_{\infty,\eps}(\CR{}(p),\CR{}(q)) = d_\eps(p,q)$ as desired.
\end{proof}






\subsection{Examples of sufficient or necessary conditions for \texorpdfstring{$\flie{\eps}$}{BLIE} spaces}
In this section we are focusing on Riemannian manifolds  with a closed measurement set and aim to extract some geometric properties of the $\flie{\eps}$-manifolds. The role of the two subsequent lemmas is to provide a partial characterization of $\flie{\eps}$-manifolds.


The following lemma provides a necessary condition for the $\flie{\eps}$ property.

\begin{lem}
\label{lem:lenght_of_geos}
Let $(M,g)$ be a smooth compact Riemannian manifold 
\begin{itemize}
    \item[(a)] with a smooth boundary $S$, or

    \item[(b)] without a boundary but having having a closed measurement set $S \subset M$.
\end{itemize}
If $(M,g,S)$ satisfies the  $\flie{\eps}$-property for some $\eps>0$ then the set $M\setminus S$ contains no geodesics that are longer than 
$2\diam(M)$.
\end{lem}

\begin{proof}
For a point $p\in M \setminus S$ we choose any unit vector $v \in S_pM$ and consider the respective maximal geodesic $\gamma=\gamma_{p,v}$ of $(M,g)$. Then, we choose $0<r<\eps$ such that the geodesic $\gamma\colon [0,r] \to M \setminus S$ is the unique distance minimizing curve between $p$ and $q:=\gamma(r)$. Thus, $d(p,q)=r<\eps$ and by the $\flie{\eps}$-property there is $t_0>0$ such that $\gamma(t_0)$ or $\gamma(-t_0)$ is in the set $S$. In particular the curve $\gamma$ meets the set $S$ in a finite time to at least one direction. 

After a reparametrization we may assume that $\gamma(0)$ is in $S$, $p=\gamma(t_0)$ for some $t_0>0$ and $\gamma((0,t_0])\subset M\setminus S$. 
Let $T:=\diam(M)$.
If $\gamma(t) \in S$ for some $t \in (t_0,T],$  we are done. Thus, without loss of generality we may suppose that there is $T_0>T$ such that $\gamma((0,T_0)) \subset M \setminus S$. Then, we consider a point $\tilde p:=\gamma(T)$, and let $0< r<\min\{\eps,T_0-T\}$ be so small that $\gamma$ is the unique distance minimizing curve between $\tilde p$ and $\tilde q=\gamma(T+r)$. Then $d(\tilde p,  \tilde q)=r<\eps$ and by the $\flie{\eps}$-property, and the assumption $\gamma((0,T_0))$ is contained in $M\setminus S$, we must have that  $\gamma|_{[0,T+r]}$ is a distance minimizing unit speed curve, or there is $T_1\geq T$ such that $\gamma(T_1) \in S$ and $\gamma|_{[T,T_1]}$ is a distance minimizing curve. 
The first case is impossible since $T+r$ is larger than the diameter of $M$. Thus, the second case is valid and it implies that $T_1\leq 2T$. Therefore, no segment of the geodesic $\gamma$ which lies in $M\setminus S$ is longer than
$
2\diam(M)
$.

Since $\gamma$ is a generic geodesic partially contained in $M\setminus S$ we have proved the claim.
\end{proof}

\begin{rem}
    If $S\subset \bS$ is a closure of an open set such that $\bS \setminus S$
    contains a great circle, then by Lemma \ref{lem:lenght_of_geos}  the manifold $\bS$ with a measurement set $S$ does not satisfy the $\flie{\eps}$-property for any $\eps>0$. The same is true for a cylinder $\mathbb S\times [0,1]$ if the measurement set is the boundary $(\bS \times \{0\})\cup (\bS \times \{1\})$. 
\end{rem}

Let $(N,g)$ be a complete Riemannian manifold and $M\subset N$ an open bounded set with a smooth boundary. We say that the set $\overline{M}$ is \textit{geodesically convex} if for each pair of points of $\overline{M}$ there exists an $N$-distance minimizing curves between these points contained in $\overline{M}$. We would like to emphasize that this type of convexity does not imply that manifold $(M,g)$ would be a simple. For instance, if we consider the paraboloid 
\[
N=\{(\cos(v)u,\sin(v)u,u^2) \in \R^3: \: v\in [0,2\pi], \: u \in \R \},
\] 
and $T>0$ then the parabolic frustum
$
M:= N \cap \{x \in \R^3: \: x_3<  T\}
$
has a strictly convex boundary $\p M =\{x \in N: \: x_3=T\}$ and $\overline{M}$ is geodesically convex in the previous sense. Also, for $T$ large enough the manifold $M$ is not simple since it has self-intersecting geodesics.

After these preparations we are ready to establish a sufficient condition for the $\flie{\eps}$-property of the Riemannian manifold $(\overline{M},g)$, when the measurement set $S$ is the boundary of $M$. This can be seen as the partial converse of Lemma \ref{lem:lenght_of_geos}.


\begin{lem}
\label{lem:converse_lenght_of_geos}
    Let $(N,g)$ be a complete  Riemannian manifold with an injectivity radius 
    $i(N,g)>0$, and $M\subset N$ an open bounded set with a smooth boundary. 
    Suppose that $\overline M$ is geodesically convex. If there is a number $L \in (0, 2i(N,g))$ such that for any geodesics of $(N,g)$ the length of no connected component of this geodesic contained in $M$ exceeds the number $L$ then $(M,g)$ with the measurement set $\p M$ satisfies the $\flie{\eps}$-property for every $\eps \in (0, i(N,g)-\frac{L}{2})$.  
\end{lem}
\begin{proof}
Suppose that the condition of lemma is valid and take any $\eps \in (0, i(N,g)-\frac{L}{2})$. 
Let $p,q \in \overline M$ be such that $d(p,q)<\eps$. Since  $\overline M$ is convex there exists a unique unit speed  geodesic $\gamma$ of $(N,g)$ connecting $p$ to $q$ with length $L(\gamma)\leq L$. Furthermore, there are numbers $0\leq t\leq s\leq L(\gamma)$ such that
\[
\gamma(t)=p, \:  \gamma(s)=q, \: s-t=d(p,q), \text{ and } \gamma(0),\gamma(L(\gamma))\in \p M.
\]
If $t,s\leq \frac{L(\gamma)}{2}$ or $t,s\geq \frac{L(\gamma)}{2}$ we have from the assumption  $L(\gamma)<2i(M,g)$ that one of the geodesic segments $\gamma|_{[0,s]}$ or $\gamma|_{[t,L(\gamma)]}$ is shorter than $i(M,g)$. Thus, at least one of these segments satisfies the $\flie{\eps}$-definition.

If $t<\frac{L(\gamma)}{2}<s$ then 
\[
s=t+(s-t)<\frac{L(\gamma)}{2}+\eps < \frac{L}{2}+i(N,g)-\frac{L}{2}=i(N,g).
\]
Thus, the geodesic segment $\gamma|_{[0,s]}$ satisfies the $\flie{\eps}$-definition.
\end{proof}

We provide some concrete examples of $\flie{\eps}$-spaces in the Appendix \ref{Sec:Appendix}.

\subsection{When is the travel time map a global isometry?}

In the main result of this subsection, Proposition \ref{prop:global_isometry}, we first derive three equivalent conditions for the travel time map being a global isometry on a compact length space with a closed measurement set. Then we turn our attention to compact Riemannian manifolds with boundary, and study two mutually exclusive boundary assumptions: 
\begin{itemize}
    \item[(i)] the boundary is strictly convex, 
    \item[(ii)] the boundary is \textit{totally geodesic}. 
\end{itemize}
The first case means that the second fundamental form of the boundary is positive definite. A submanifold is called totally geodesic if any geodesic on the submanifold, with its induced Riemannian metric, is also a geodesic of the ambient manifold. When the measurement set is the boundary, satisfying the assumption (i) or (ii), we provide in Proposition \ref{prop:global_isometry} two additional conditions which are also equivalent to the isometry property of the travel time map. The main observation of this subsection is that a compact Riemannian manifold with a strictly convex boundary and an isometric travel time mapping needs to be essentially simple. The converse result was observed in \cite[Proposition 6]{ilmavirta2023three}.

\begin{prop} \label{prop:global_isometry}
Let $(X,d)$ be a compact length space with a closed measurement set $S$. The following conditions are equivalent:
    \begin{enumerate}
        \item[(A)] $\CR{}\colon (X,d) \to (C(S), \|\cdot\|_\infty)$ is a metric isometry.
        
        \item[(B)] For all $p,q\in X$ there exist $z \in S$ and a distance minimizing curve $\sigma:[0,L]\to X$ from $p$ to $z$ that goes through $q$ (or from $q$ to $z$ via $p$).
        
        \item[(C)] Any distance minimizing curve between two points of $X\setminus S$ has a distance minimizing extension forwards or backwards.
    \end{enumerate}
    Moreover, if $(X,d) = (M,g)$ is a compact Riemannian manifold with nonempty boundary $\partial M=S$ which is strictly convex or totally geodesic, then the following two properties are also equivalent with the former three conditions:
    \begin{enumerate}
        \item[(D)] Every pair of interior points of $M$ is connected by a unique geodesic which is distance minimizing all the way to the boundary (not necessarily both ways).
        \item[(E)] The Riemannian manifold $(M,g)$ 
            \begin{itemize}
                \item[(i)] is non-trapping,
                \item[(ii)] has no interior conjugate points, in the sense that for no pair of interior points is there a geodesics connecting these points along which the points are conjugate to each other, and
                \item[(iii)] each pair of its interior points are connected by a unique distance minimizing geodesic whose trace is contained in the interior.
        \end{itemize}
    \end{enumerate}
\end{prop}

\begin{proof}
        (A)$\implies$(B): By assumption for all $p,q\in X$ we have that 
        $$
        \|\CR{}(p) - \CR{}(q)\|_\infty  = \sup_{z\in S} |d(p,z) - d(q,z)| = d(p,q).
        $$
        By compactness of the set $S$ there exists a point $z\in S$ such that either $d(p,z) = d(p,q) + d(q,z)$ or we have that $d(q,z) = d(q,p) + d(p,z)$. Without loss of generality we assume that the first case is true. Since $X$ is a length space the points $p$ and $q$ and $q$ and $z$ can be connected by distance minimizing curves. Thus, by the former equality, the composition of these curves is a distance minimizing curve from $p$ to $z$ going through $q$.
        
        
        (B)$\implies$(A):  Let $p,q \in X$. Let $z \in S$ and $\sigma:[0,L]\to X$ be a distance minimizing curve  from $p$ to $z$ that goes through $q$. Then 
        $$
        d(p,q) = |d(p,z) - d(q,z)| \leq \sup_{x\in S} |d(p,x) - d(q,x)| = \|\CR{}(p) - \CR{}(q)\|_\infty
        $$
        and since $\CR{}$ is $1$-Lipschitz we have that $d(p,q) = \|\CR{}(p) - \CR{}(q)\|_\infty$. Thus $\CR{}$ is a metric isometry.
        

        
        (B)$\implies$(C):  Let $p,q\in X\backslash S$. Let $\ell=d(p,q)$ and $\sigma:[0,\ell]\to X$ be a distance minimizing curve connecting $p$ to $q$.  By the hypothesis we may take $z \in S$ and a distance minimizing curve $\tilde{\sigma}:[0,L]\to X$, parameterized by the arc length, which goes from $p$ to $z$ via $q$. Thus, the concatenation of $
        \tilde \sigma|_{[\ell,L]}$ and $\sigma
        $ is a distance minimizing extension of $\sigma$.


        (C)$\implies$(B): %
        Let $p,q\in X\backslash S$ and let $\sigma:[0,\ell]\to X$ be a distance minimizing curve connecting $p$ to $q$. We use Zorn's lemma to show that this curve has a maximal distance minimizing extension. For this we let $E$ to be the set of all distance minimizing extensions of $\sigma$, parameterized by their arc length. On the collection $E$ we introduce a partial ordering: if the curves $\sigma_1$ and $\sigma_2$ are in $E$, we say that $\sigma_1 \leq \sigma_2$ if 
        $\sigma_2$ is a distance minimizing extension of $\sigma_1$. 
        Let $C \subset E$ be a chain with respect to our partial ordering, that is if $\sigma_1,\sigma_2 \in C$ then one of them is an extension of the other. To apply Zorn's Lemma we need to show that $C$ has an upper bound in $E$. We omit the rest of the details as they are standard results in a compact length space.

        Let $\tilde\sigma \colon [a,b]\to X$ be any maximal distance minimizing extension of $\sigma \colon [0,\ell]\to X$. If both $\tilde\sigma(a)$ and $\tilde\sigma(b)$ are contained in $X\backslash S$, then by condition (B) there must exist a distance minimizing extension of $\tilde\sigma$. But $\tilde\sigma$ was assumed to be a maximal distance minimizing extension of $\sigma$, this is a contradiction. We conclude that $\tilde\sigma(a)$ or $\tilde\sigma(b)$ must be in $S$ and finish our proof that (C) implies (B).


        We have proved that (A), (B) and (C) are equivalent. From here onwards we assume that $(X,d)=(M,g)$ is a smooth Riemannian manifold with boundary $\p M=S$, which is either strictly convex or totally geodesic. We show that (A) and (D) are equivalent.
        
        (D)$\implies$(A): 
        Clearly with the choice $S=\p M$ the condition (D) implies (B) and (A). 
        
        (A)$\implies$(D): Let $p,q\in M$ be interior points. As $\CR{M,\p M}$ is a metric isometry, then there exists a distance minimizing curve $\sigma$ starting from $p$, going through $q$ and ending at some boundary point $z$. For the sake of contradiction assume that $\tilde{\sigma}$ is another distance minimizing curve from $p$ to $q$ ending at a boundary point $\tilde{z}$. We consider the case when $\p M$ is strictly convex and totally geodesic separately.
          
          
We assume that the boundary is strictly convex, then the curves $\sigma$ and $\tilde \sigma$ are both geodesics of $M$ and they are contained in the interior of $M$ modulo the end points. As they both go through $p$ and $q$ while being different, we can use these curves to construct a distance minimizing curve (also a geodesic) from $q$ to $z$ which is not $C^1$ at some point. This is a contradiction since geodesics are always smooth curves.



Then we assume that the boundary is totally geodesic. It was shown in \cite[Chapter 4, Proposition 13]{o1983semi} that the following three conditions are equivalent:
\begin{itemize}
    \item a submanifold is totally geodesic,
    \item the second fundamental form of the submanifold vanishes identically,
    \item each geodesic of the ambient manifold which is tangential to the submanifold stays in the submanifold for a short time. 
\end{itemize}
It follows from \cite{alexander1981geodesics} that the distance minimizing curves $\sigma$ and $\tilde \sigma$ are $C^1$-smooth, have only tangential intersections with the boundary (in the interior of their domain) and outside the boundary these curves are geodesic of $(M,g)$. As the points $p$ and $q$ lie in the interior we get from the total geodesicity of the boundary that the curves $\sigma$ and $\tilde \sigma$ can only have a transversal intersection with boundary, and this must happen at the terminal points. Therefore curves $\sigma$ and $\tilde \sigma$ are geodesics of $(M,g)$, which are contained in the interior of $M$ modulo the end points. Thus, we arrive in a contradiction by the same argument as in the previous part.


Finally, we show that (A) and (E) are equivalent. 

(A)$\implies$(E): If (A) is true then also (E) is true. Due to Lemma \ref{lem:lenght_of_geos} it holds that $(M,g)$ is non-trapping. Let $p,q$ be some interior points, and choose the unique distance minimizing geodesic between $p$ and $q$ which ends at the boundary. By \cite[Proposition 10.32]{lee2018introduction} no geodesic is minimizing beyond the first conjugate point. Therefore, the points $p$ and $q$ cannot be conjugate along $\gamma$. We have verified the conditions conditions (i) -- (iii) of property (E).

(E)$\implies$(A):  Let $p,q \in M$ be some interior points. Then there is a unique geodesic $\gamma\colon [0,L] \to M$ such that $\gamma(L) \in \p M$, $\gamma((0,L))\subset M^\interior$, $\gamma(0)=q$ and $\gamma(t_0)=p$ for $t_0=d(p,q) \in (0,L)$. We claim that $\gamma$ is a distance minimizer in the segment $[0,L]$. This implies condition (A).

We define 
\[
t_\star:=\sup\{t\in [0,L]: \: d(\gamma(0),\gamma(t))=t\}.
\]
Clearly $t_\star \in [t_0,L]$. If $t_\star< L$ then we get from the boundary assumptions and Klingenbergs lemma \cite[10.32]{lee2018introduction} (see also \cite[Lemma 3.2]{pavlechko2022uniqueness}) that at least one of the following conditions is true:
\begin{itemize}
    \item The points $\gamma(0)$ and $\gamma(t_\star)$ are conjugate along $\gamma$,
    \item There is another distance minimizing geodesic from $\gamma(0)$ to $\gamma(t_\star)$.
\end{itemize}
Since $q=\gamma(0)$ and $\gamma(t_\star)$ are interior points either of these conditions violates the assumptions of condition (E). Thus $t_\star=L$ as claimed. 
\end{proof}

To close this section we note that a closed manifold does not have a canonical choice of the measurement set. Furthermore, it is easy to show that even for very large measurement sets the travel time map might or might not be a global isometry. For instance consider $\mathbb{S}^2$ with the measurement set 
    \[
    S=\mathbb{S}^2 \setminus (B(e_3,r)\cup B(-e_3,r)), \quad  \text{where }r \in (0, \pi/2).
    \]
    Then the north and south poles $e_3$ and $-e_3$ are contained in the exterior of $S$ and no geodesic between them can be extended as distance minimizer to some point in $S$. Thus, the respective travel time map $\CRR{\mathbb{S}^2, S}$ is not a global isometry. 
    
    On the contrary with the choice of the measurement set
    \[
    S=\mathbb{S}^2 \setminus B(e_3,r), \quad  \text{ where } r\in (0,\pi /2 ),
    \]
    the  travel time map $\CRR{\mathbb{S}^2, S}$ is a global isometry since the polar cap $B(e_3,r)$ is a simple manifold.

    Clearly the northern hemisphere
$
    M=\{x \in \mathbb{S}^{n-1}: \: x_n\geq 0\}
$
    is a smooth compact Riemannian manifold with a totally geodesic boundary
$
    \p M = \{x \in \mathbb{S}^{n-1}: \: x_n= 0\}.
$
    On $M$ no pair of interior points are antipodal to each other. Thus, they are connected by a unique great circle which intersects the boundary at exactly two points having the length $\pi$. Therefore, $M$ satisfies condition (B) of Proposition \ref{prop:global_isometry}. In particular, the map $\CRR{M,\p M}$ can be a global isometry even if there are conjugate points between boundary points, but not if there are interior conjugate points as we showed above.

\section{Truncated Compact Length Spaces}
\label{sec:length_Spaces}
In this section we survey some properties of compact length spaces and show that the travel time map of any $\flie{\eps}$-space preserves the length structure. We end the section by providing a list of many important properties between distances and their truncations.

\subsection{The travel time map of \texorpdfstring{$\flie{\eps}$}{AE}-space preserves the length structure}

There is no \textit{a priori} reason to assume that for a general compact length space $(X,d)$ with a closed measurement set $S \subset X$ the space $(\CR{}(X),\|\cdot\|_\infty)$ would need to be a length space, even though the ambient Banach space $(C(S), \|\cdot\|_\infty)$ is one. The following lemma will show that the length metric $d_L$ of $\CR{}(X)$, which is induced by the supremum metric, is finite.

\begin{lem}
Let $(X,d)$ be a compact length space with a closed measurement set $S\subset X$. If the travel time map $\CR{}$ is a topolological embedding then the length metric $d_{\infty,L}$ of $\CR{}(X)$, which is induced by the supremum metric, is finite. In particular, every pair of points $x,y \in \CR{}(X)$ is connected by a curve of finite length.
\end{lem}
\begin{proof}
For $x,y \in \CR{}(X)$ we consider the inverse images $p:=\CR{}^{-1}(x), \: q:=\CR{}^{-1}(y) \in X$. Since $(X,d)$ is a length space there is a finite length curve $\gamma\colon [0,1]\to X$ such that $\gamma(0)=p$ and $\gamma(1)=q$. Thus, the curve $\tilde \gamma :=\CR{} \circ \gamma\colon [0,1] \to \CR{}(X)$ is continuous, begins at $x$ and ends at $y$. Furthermore, since the travel time map is a contraction we have for any
\[
0=t_0<t_1<t_2<\ldots <t_{N-1}<t_N=1,
\]
that
\[
\|\tilde \gamma(t_{i-1})-\tilde \gamma(t_i)\|_\infty\leq d(\gamma(t_{i-1}),\gamma(t_i)), \quad \text{for all } i \in \{1,\ldots,N\}.
\]
Since the curve $\gamma$ has a finite length, the former estimate implies that also the curve $\tilde \gamma$ has a finite length. 

We have proved that all points of $\CR{}(X)$ can be connected by a curve of finite length contained in $\CR{}(X)$. Therefore, $d_{\infty,L}$ is finite. 
\end{proof}

\begin{prop}
\label{thm:isometry_of_length_spaces}
    If a compact length space $(X,d)$ with a closed measurement set $S\subset X$ satisfies the $\flie{\eps}$ property for some $\eps>0$ then the travel time map $\CR{} \colon (X,d) \to (\CR{}(X),d_{\infty,L})$ is a bijective metric isometry.
\end{prop}
\begin{proof}
We show first that the map $\CR{}$ preserves the lengths of curves. Let $\gamma\colon [0,1] \to (X,d)$ be a finite length curve, and denote $\tilde \gamma :=\CR{} \circ \gamma$.
Since the curve $\gamma$ is uniformly continuous there is $\delta>0$ such that whenever the times $t,s\in [0,1]$ are $\delta$-close then the points $\gamma(t),\gamma(s)$ are $\eps$-close in $(X,d)$.
Let 
\[
0=t_0<t_1<t_2<\ldots <t_{N-1}<t_N=1,
\]
be a partition of $[0,1]$ such that
\[
t_i-t_{i-1}<\delta, \quad \text{ for all } i \in \{1,\ldots,N\}.
\]
Then by the $\flie{\eps}$ property we have that
\begin{equation}
\label{eq:length_preserve}
\sum_{i=1}^Nd(\gamma(t_{i-1}),\gamma(t_i))=\sum_{i=1}^N\|\tilde \gamma(t_{i-1})-\tilde \gamma(t_i)\|_\infty.
\end{equation}
Furthermore, if $t \in (t_{i-1},t_i)$ for some $i \in \{1,\ldots,N\}$ then both $t-t_{i-1}$ and $t_i-t$ are smaller than $\delta$. Thus, we have that
\[
d(\gamma(t_{i-1}),\gamma(t))=\|\tilde \gamma(t_{i-1})-\tilde \gamma(t)\|_\infty, 
\quad \text{and} \quad
d(\gamma(t_{i}),\gamma(t))=\|\tilde \gamma(t_{i})-\tilde \gamma(t)\|_\infty,
\]
and in conjunction with Equation \eqref{eq:length_preserve} we get that $L_d(\gamma)=L_\infty(\tilde \gamma)$.\\ \newline
Let $p,q \in X$ 
and denote $x:=\CR{}(p)$ to $y:=\CR{}(q)$.
Since $(\CR{}(X),d_{\infty,L})$ is a compact length space it follows from Hopf–Rinow–Cohn-Vossen Theorem (c.f \cite[Theorem 2.5.28]{burago2001course}) that there is a $d_{\infty,L}$-distance minimizing curve $\tilde \sigma \colon [0,1] \to \CR{}(X)$ from $x$ to $y$. Since the map $\CR{}$ is $\eps$-local isometry it follows from \cite[Lemma 3.4.17]{burago2001course} that there exits a unique $d$-distance minimizing curve $\sigma\colon [0,1] \to (X,d)$ starting from $p$ such that
\[
\CR{}(\sigma(t))=\tilde \sigma(t) \quad \text{ for all } t\in [0,1].
\]
Since the map $\CR{}$ is one-to-one and 
$
\tilde \sigma(1)=\CR{}(q)
$
we must have that $\sigma(1)=q$. Thus, $\sigma$ is $d$-distance minimizing curve from $p$ to $q$, and 
due to the first part of the proof we get
\[
d(p,q)=L_d(\sigma)=L_\infty(\tilde \sigma)=d_{\infty,L}(x,y).
\]
This concludes the proof.
%
\end{proof}

\subsection{Properties of the truncation}

A \textit{correspondence} between any two sets $X$ and $Y$ is a subset of $X\times Y$ satisfying the following two conditions:
\begin{itemize}
    \item[(I)] for any $x \in X$ there is $y \in Y$ such that $(x,y) \in \mathfrak{R}$ 
    \item[(II)] for any $y \in Y$ there is $x \in X$ such that $(x,y) \in \mathfrak{R}$. 
\end{itemize}
If the sets $X$ and $Y$ are equipped with metrics $d_X$ and $d_Y$ respectively then the \textit{distortion} of the correspondence $\mathfrak{R}$ is the number
\[
\text{dis}(\mathfrak{R})=\sup\{|d_X(x,x')-d_Y(y,y')|: \: (x,y), \: (x',y') \in \mathfrak{R}\}.
\]

In particular, if $(X,d_X)$ and $(Y,d_Y)$ are compact metric spaces, then it was proven in \cite[Theorem 7.3.24.]{burago2001course} that     
\begin{equation}
\label{eq:GH_with_correspondece}
    \ghd{}((X,d_X),(Y,d_Y))=\frac{1}{2} \inf_{\mathfrak{R}}(\text{dis}(\mathfrak{R})),
\end{equation}
where the infimum is taken over all correspondences between metric spaces $(X,d_X)$ and $(Y,d_Y)$. 

The following lemma summaries some properties between a metric $d$ of a compact metric space $X$ and its $\eps$-truncation $d_\eps$, which was given in Definition \ref{defn:truncation}.

\begin{lem}
\label{lem:props_of_trunc}
Let $(X,d_X)$ be a compact metric space and $\eps>0$. Then
\begin{enumerate}
    \item
    The topologies of $d_X$ and the $\eps$-truncated metric $d_{X,\eps}$ coincide.
    \item
    The length structures induced by $d_X$ and $d_{X,\eps}$ coincide.
    \item
    Let $K_1,K_2 \subset X$ be closed. If $\hd{X,d_X}(K_1,K_2)<\eps$ then 
\begin{equation}
    \label{eq:Hausdorff_equality}
    \hd{X,d_{X,\eps}}(K_1,K_2)=\hd{X,d_X}(K_1,K_2).
\end{equation}
\item
If $(Y,d_Y)$ is a compact metric space and $f\colon (X,d_X) \to (Y,d_Y)$ is a metric isometry then $f$ is also a metric isometry between the truncated spaces $(X,d_{X,\eps})$ and $(Y,d_{Y,\eps})$.
\item
We always have that
\begin{equation}
\label{eq:GH_truncation_1}
\ghd{}((X,d),(X,d_\eps))=\frac{1}{2} \max\{\diam(X)-\eps, 0\}
\end{equation}
\item
If $(X,d_X)$ and $(Y,d_Y)$ are compact length spaces such that the respective truncated spaces $(X,d_{X,\eps})$ and $(Y,d_{Y,\eps})$ are isometric then the original spaces $(X,d_X)$ and $(Y,d_Y)$ are also isometric.
\end{enumerate}
\end{lem}

\begin{proof}
1.
    Clearly all small balls of metrics $d_X$ and $d_{X,\eps}$ agree as sets. Thus, the first claim follows from the definition of the open sets in metric spaces.

2.
    If $\gamma\colon [0,1] \to X$ is a curve then we see by a similar argument as in the proof of Theorem \ref{thm:isometry_of_length_spaces} that the length of $\gamma$ is the same with respect to both metrics $d_X$ and $d_{X,\eps}$. Thus, the length structures induced by $d_X$ and $d_{X,\eps}$ coincide.

3.
    Let $K_1,K_2 \subset X$ be closed and $d^{H}(K_1,K_2)<\eps$. We denote $d^{H}(K_1,K_2):=\delta$. Then for each $\tilde \eps \in (\delta,\eps)$ it holds that
    \[
    K_2 \subset \bigcup_{x \in K_1} B_d(x,\tilde \eps)=\bigcup_{x \in K_1} B_{d_\eps}(x,\tilde \eps), \quad \text{ and } \quad
    K_1 \subset \bigcup_{y \in K_2} B_d(y,\tilde \eps)=\bigcup_{y \in K_2} B_{d_\eps}(y,\tilde \eps). 
    \]
    Thus, we get from the definition of the Hausdorff distance that
    \[
    \hd{X,d_{X,\eps}}(K_1,K_2)\leq \delta.
    \]
    On the other hand by the definition Hausdorff distance we may without loss of generality assume that for each $j \in \N$ large enough we have that
    \[
     K_1 \setminus U_j \neq \emptyset, \quad \text{where }U_j:=\bigcup_{y \in K_2} B_d
    \left(y,\tilde \delta-\frac{1}{j}\right)=\bigcup_{y \in K_2} B_{d_\eps}\left(y,\delta-\frac{1}{j}\right). 
    \]
    Thus,
    \[
    \hd{X,d_\eps}(K_1,K_2)\geq  \delta.
    \]
    We have verified the equality~\eqref{eq:Hausdorff_equality}.

4.
    Since the map $f\colon (X,d_X) \to (Y,d_Y)$ is a metric isometry we have
    \[
    d_X(p,q)=d_Y(f(p),f(q)), \quad \text{for all } p,q \in X.
    \]
    Therefore $d_X(p,q)\leq \eps$ if and only if $d_Y(f(p),f(q))<\eps$. Therefore,
    \[
    d_{X,\eps}(p,q)=d_{Y,\eps}(f(p),f(q)), \quad \text{for all } p,q \in X.
    \]

5.
    Suppose first that $\eps\leq\diam(X)$. Then by \cite[Exercise 7.3.14]{burago2001course} we have that 
    \[
    \ghd{}((X,d),(X,d_\eps))\geq \frac{1}{2}|\diam(X,d)-\diam(X,d_\eps)|=\frac{1}{2}(\diam(X,d)-\eps).
    \]
    On the other hand the diagonal set 
    \[
    \mathfrak{R}:=\{(x,x):\: x \in X\}
    \]
    is a  correspondence between the spaces $(X,d)$ and $(X,d_\eps)$ whose distortion is 
    \[
    \text{dis}(\mathfrak{R})=\sup_{x,x' \in X}|d(x,x')-d_\eps(x,x')|=\diam(X)-\eps.
    \]
    Hence, equation \eqref{eq:GH_with_correspondece} implies
    \[
    \ghd{}((X,d),(X,d_\eps)) = \frac{1}{2}(\diam(X)-\eps).
    \]
    in this case. If $\eps > \diam(X)$ then $d=d_\eps$ and the identity map of $X$ is an isometry. Thus, equation~\eqref{eq:GH_truncation_1} is true.

6.
    If $f\colon (X,d_{X,\eps}) \to (Y,d_{Y,\eps})$ is a bijective isometry then the same map between the original length spaces is a $\eps$-local isometry. Thus, by earlier parts of this lemma, and the proof of Proposition \ref{thm:isometry_of_length_spaces} the map $f\colon (X,d_X) \to (Y,d_Y)$ is a metric isometry. 
\end{proof}

\begin{rem}
\label{rem:non_isometric_trunc}
By the previous lemma the truncated Gromov--Hausdorff-distance is a valid metric in the space of compact length spaces, no matter what $\eps>0$ is chosen. However, this is not true if we drop the length space assumption as the following example illustrates. The Gromov--Hausdorff distance between the sets $X=\{0,1\}$ and $Y=\{0,2\}$ with the obvious metrics is $\frac12$, but $\ghd{}(X_\eps,Y_\eps)=0$ for all $\eps\in (1/2,1)$.

\end{rem}

\section{Verification of the \texorpdfstring{$\blie{\eps}$}{BE}-property from the travel time data of a \texorpdfstring{$\flie{\eps}$}{AE}-space}
\label{sec:verification}
By Proposition \ref{prop:(eps,delta)_isometry} we know that each $\flie{\eps}$-space $(X,d)$ with a closed measurement set $S\subset X$ is also a $\blie{\eps}$-space after possible choosing a smaller $\eps$. However, the proof is based on a compactness argument and does not directly give us a method to show that $\cR^{-1}_{X,S}$ is a $\eps$-local isometry based on the travel time data $\CR{}(X)$. It was shown in \cite[Section 2.4.]{burago2001course} that a complete metric space is a length space if and only if every pair of points has a midpoint. For our purposes we introduce the following generalization of the midpoint property.


\begin{defn}
    Let $\eps > 0$ and let $(X,d)$ be a metric space. We say that $(X,d)$ satisfies the $\eps$-local midpoint property if for all $x,y\in X$ with $d(x,y) < \eps$, there exists $m \in X$ such that
    $d(m,x) = d(m,y) =\frac{1}{2}d(x,y)$.
\end{defn}

According to the following proposition the $\eps$-local isometry property of the inverse map $\cR^{-1}_{X,S}$ is equivalent to the $\eps$-local midpoint property of the travel time data and $\eps$-local isometry property of the travel time map $\CR{}$. Furthermore, if we know \textit{a priori} that $(X,d)$ and $S$ satisfy the $\flie{\eps_0}$ for some $\eps_0>0$ we can find the best $\eps>0$ for which the $\eps$-local midpoint property is valid for the travel time data. Then, the space $(X,d)$ and $S$ need to automatically be in both $\flie{\eps}$- and $\blie{\eps}$-classes.

\begin{prop}
 \label{thm:delta_test}
    Let $\eps > 0$ and let $(X,d)$ be a compact length space with a closed measurement set $S \subset X$ such that the travel time map $\CR{}$ is a topological embedding. The following conditions are equivalent. 
    \begin{enumerate}
        \item $(X,d)$ with measurement set $S$ is $\blie{\eps}$
        \item The space $(X,d)$ with the set $S$ satisfy the $\flie{\eps}$-property and the compact metric space $\CR{}(X)$ with the supremum norm satisfies the $\eps$-local midpoint property.
        \item $(X,d)$ with the set $S$ satisfy the $\flie{\eps_0}$ property for some $\eps_0 \in (0,\eps)$ and $\CR{}(X)$ with the supremum norm satisfies the $\eps$-local midpoint property.
    \end{enumerate}
\end{prop}

\begin{proof}
1$\implies$2:
    We assume that the map $\CR{}^{-1}:(\CR{}(X),\|\cdot\|_\infty)\to (X,d)$ is a $\eps$-local metric isometry and show first that also the forward map $\CR{}$ is a $\eps$-local isometry. Let $p,q \in X$, and suppose that $d(p,q) < \eps$. Since the map $\CR{}$ is nonexpansive and we have that
    \[
    \|\CR{}(p)-\CR{}(q)\|_\infty\leq d(p,q) < \eps.
    \]
    Thus, $\|\CR{}(p)-\CR{}(q)\|_\infty = d(p,q)$. This is precisely the condition for $\CR{}$ to be a $\eps$-local isometry. 
    
    We now show that the compact metric space $(\CR{}(X),\|\cdot\|_\infty)$ satisfies the $\eps$-local midpoint property. Let $p,q\in X$ be such that $\| \CR{}(p)-\CR{}(q)\| < \eps$, then $d(p,q) = \|\CR{}(p)-\CR{}(q)\|_\infty$. Since $(X,d)$ is a length space it satisfies the midpoint property globally per  \cite[Lemma 2.4.8.]{burago2001course}. Thus there exists a point $m \in X$ such that $d(p,m) = d(q,m) = \frac{1}{2}d(p,q)$. Furthermore, $d(p,q) = d(p,m) + d(m,q)$ so we observe that since $\CR{}$ is a nonexpansive map we have that
    \begin{align*}
        &\|\CR{}(p)-\CR{}(m)\|_\infty + \|\CR{}(m)-\CR{}(q)\|_\infty 
        \\
        &\leq d(p,m)+d(m,q) = d(p,q) = \|\CR{}(p)-\CR{}(q)\|_\infty < \eps
    \end{align*}
    Thus $\|\CR{}(p)-\CR{}(m)\|_\infty, \|\CR{}(m)-\CR{}(q)\|_\infty < \eps$, and by the local $\eps$-local isometry property we have that 
    $$
    \|\CR{}(p)-\CR{}(m)\|_\infty = d(p,m) = \frac{1}{2}d(p,q) = \frac{1}{2}\|\CR{}(p)-\CR{}(q)\|_\infty,
    $$
    and similarly $\|\CR{}(q)-\CR{}(m)\|_\infty = \frac{1}{2}\| \CR{}(p)-\CR{}(q)\|$. Hence, $\CR{}(m)\in \CR{}(X)$ is a midpoint of $\CR{}(p)$ and $\CR{}(q)$. 

2$\implies$3:
    Condition 2 is stronger than condition 3. 

3$\implies$1:
    Suppose that the compact metric space $(\CR{}(X),\|\cdot\|_\infty)$ satisfies the $\eps$-local midpoint property. Let $\CR{}(p),\CR{}(q)\in \CR{}(X)$ be such that $\|\CR{}(p)-\CR{}(q)\|_\infty<\eps$. A number $0\leq \frac{k}{2^n} \leq 1$ for some $k,n \in \N$ is called a dyadic rational, and the set $\cD$ of these rationals form a dense subset of $[0,1]$. From here we utilize the $\eps$-local midpoint property and follow the steps in the proof of \cite[Theorem 2.4.16]{burago2001course} to find a function $\gamma \colon \cD \to \CR{}(X)$
which satisfies the following Lipschitz property
\begin{equation}
    \label{eq:Lip_cont_gamma}
\|\gamma(t)-\gamma(t')\|_\infty\leq |t-t'|\|\gamma(0)-\gamma(1)\|_\infty, \quad \text{ for any } t,t' \in \cD.
\end{equation}
Since the metric space 
$(\CR{}(X), \|\cdot\|_\infty)$ is complete we get from \cite[Proposition 1.5.9.]{burago2001course} that the map $\gamma$ has a unique Lipschitz continuous extension to $[0,1]$. 
Therefore, $\gamma\colon [0,1] \to \CR{}(X)$ is a Lipschitz continuous path connecting $\CR{}(p)$ to $\CR{}(q)$. Furthermore, the estimate \eqref{eq:Lip_cont_gamma} implies that 
\[
L_\infty(\gamma)=\|\CR{}(q)-\CR{}(q)\|_\infty.
\]
Since the space $(X,d)$ with the measurement set $S$ satisfy the $\flie{\eps_0}$-property it follows from Proposition \ref{thm:isometry_of_length_spaces} that the map $\CR{}$ is a metric isometry between the length spaces $(X,d)$ and $(\CR{}(X),d_{\infty,L})$. Therefore, we have that
    \[
    d(p,q)=d_{\infty,L}(\CR{}(p),\CR{}(q))=L_\infty(\gamma)=\|\CR{}(p)-\CR{}(q)\|_\infty.
    \]
This implies that $\cR^{-1}_{X,S}$ is a $\eps$-local isometry.
\end{proof}

\section{Proofs of the Main Theorems}
\label{Sec:proofs}
Equation \eqref{eq:GH_with_correspondece} implies a well known upper bound for the Gromov--Hausdorff distance between compact metric spaces $(X,d_X)$ and $(Y,d_Y)$:
    \begin{equation}
    \label{prop:ghd_upper_bound}
        \ghd{}(X,Y) \leq
        \frac12\max\{\diam(X),\diam(Y)\}
    \end{equation}
We are ready to prove our second main result.

\begin{proof}[Proof of Theorem~\ref{thm:stability_1}]
First we consider the case when 
\begin{equation}
    \label{eq:dist_of_data_1}
\hd{C(S_1)}(\CRR{X_1,S_1}(X_1),\CRR{X_2,S_2,\phi}(X_2))
 < \eps.
\end{equation}
 By Proposition \ref{thm:local_isometry} the maps $\CRR{X_i,S_i} \colon (X_i,d_o) \to (C(S_i),\|\cdot\|_\infty)$ are $\eps$-local isometries for both $i \in \{1,2\}$. Since the map $\phi \colon S_1 \to S_2$ is a homeomorphism, the induced map
\[
\Phi\colon (C(S_2),\|\cdot\|_\infty) \to (C(S_1),\|\cdot\|_\infty), \quad \Phi(f)=f\circ \phi
\]
is an isometry. Thus, the map $\Phi \circ \CRR{X_2,S_2} \colon (X_2,d_2) \to (C(S_1),\|\cdot\|_\infty)$ is a $\eps$-isometry with the range
\[
(\Phi \circ \CRR{X_2,S_2})(X_2)=\CRR{X_2,S_2,\phi}(X_2).
\]
Moreover, we get from the Definition \ref{defn:special} and Proposition \ref{prop:(eps,delta)_isometry} that the maps $\CRR{X_1,S_1}$ and $\Phi \circ \CRR{X_2,S_2}$ are isometries from the truncated spaces $(X_1,d_{1,\eps})$ and $(X_2,d_{2,\eps})$ into the truncated complete metric space $(C(S_1),d_{\infty,\eps})$. Therefore, we get from the assumption \eqref{eq:dist_of_data_1} and Lemma \ref{lem:props_of_trunc} that 
\[
\hd{C(S_1),d_{\infty,\eps}}(\CRR{X_1,S_1}(X_1),\CRR{X_2,S_2,\phi}(X_2))
=
\hd{C(S_1),\|\cdot\|_\infty}(\CRR{X_1,S_1}(X_1),\CRR{X_2,S_2,\phi}(X_2)).
\]
Hence, the inequality \eqref{eq:dist_of_manifolds_2} follows from Definition \ref{defn:truncated_GH_dist} of the truncated Gromov--Hausdorff distance.

Then we suppose that the estimate \eqref{eq:dist_of_data_1} does not hold. Since the diameters of the truncated space $(X_i,d_{i,\eps})$ do not exceed $\eps$ we get from \eqref{prop:ghd_upper_bound} that $\ghd{\eps}((X_1,d_1),(X_2,d_2)) \leq \frac{\eps}{2}$. Thus, the inequality $\eqref{eq:dist_of_manifolds_2}$ follows trivially. 

If $\hd{C(S_1)}(\CRR{X_2,S_2}(X_1),\cR_{X_2,S_2,\phi}(X_2))=0$ then the inequality \eqref{eq:dist_of_manifolds_2} implies that the truncated spaces $(X_1,d_{1,\eps})$ and $(X_2,d_{2,\eps})$ are isometric. Thus, it follows from Lemma \ref{lem:props_of_trunc} that there is a metric isometry between the spaces $X_1$ and $X_2$. 
\end{proof}

In order to prove the uniqueness result for Finsler manifolds in Corollary \ref{cor:stability_2} we need to find a boundary adapted version of the classical Myers--Steenrod theorem \cite{myers-steenrod}: Every bijective metric isometry between complete smooth Riemannian or Finsler manifolds is a smooth Riemannian or Finslerian isometry. The Finsler version of this result was given for instance in \cite[Theorem A]{matveev2017myers}.


\begin{prop}[Myers–Steenrod Theorem for Compact Manifolds with Boundary]
 \label{prop:MS_thm}
Let $(M_1,F_1)$ and $(M_2,F_2)$ be compact, connected smooth reversible Finsler manifolds with boundary. If a map $f\colon M_1 \to M_2$ is a bijective metric isometry, then it is also a Finslerian isometry in the sense that it is a smooth map that pulls the metric $F_2$ back to $F_1$.  
\end{prop}

\begin{proof}
    We prove first that the map $f$ and its inverse are smooth. 
    To this end we note that by the invariance of domain the function $f$ maps the interior of $M_1$ onto the interior of $M_2$ and the boundary of $M_1$ onto the boundary of $M_2$. For the interior the claim follows from the proof of traditional Myers–Steenrod theorem presented for instance in \cite[Theorem 5.6.15.]{petersen2006riemannian} for Riemannian manifolds and in \cite[Theorem A]{matveev2017myers} for Finsler manifolds. Hence, it suffices to show that $f$ and $f^{-1}$ are smooth near the boundaries. 
    
    We note that for each $i \in \{1,2\}$ we can use the Finslerian distance $d_i$ of the manifold $M_i$ to induce an intrinsic distance to each connected component of the boundary $\p M_i$. In particular, this distance agrees with the Finslerian distance $\rho_i$ of the closed Finsler manifold $(\p M_i,F_i|_{\p M_i})$. Thus, the restriction map $f \colon (\p M_1,\rho_1) \to (\p M_2,\rho_2)$ is a metric isometry for the Finslerian distances on each connected component of $\p M_1$. By the Myers--Steenrod theorem the restriction map $f|_{\p M_1}\colon (\p M_1,F_1|_{\p M_1}) \to (\p M_2,F_2|_{\p M_2})$ is a smooth Finslerian isometry.

    Since $M_i$ is compact there exists $\eps>0$ such that the boundary normal exponential map 
    \[
    \exp_{\p M_i}\colon [0,\eps)\times \p M_i \to M_i, \quad \exp_{\p M_i}(s,z)=\gamma_{z,\nu_i(z)}(s),
    \]
    is a diffeomorphism onto its image \cite[Lemma 3.3]{dehoop2020determination}. Here $\gamma_{z,\nu_i(z)}$ is the geodesic of $(M_i,g_i)$ whose initial conditions are $(z,\nu_i(z))$ and $\nu_i$ is the inward pointing boundary normal vector field. The inverse $\varphi_i$ of the map $\exp_{\p M_i}$ is called the \textit{boundary normal coordinates} and for each point $p \in M_i$, which also lies in the domain of $\varphi_i$, we have that
    $
    \varphi_i(p)=(s_i(p),z_i(p)) \in [0,\eps)\times \p M_i
    $
    where $z_i(p)$ is the closest boundary point to $p$ and $s_i(p)=d_i(z_i(p),p)=d(\p M_i,p)$ is the distance from $p$ to $\p M_i$. For further information about the boundary normal coordinates we refer to \cite[Example 6.44.]{lee2018introduction}. 

    Since the restriction map $f|_{\p M_1}$ is a smooth Finslerian isometry we have that the map 
    \begin{equation}
    \label{eq:local_representation_of_f}
    [0,\eps)\times\p M_1  \ni (s,z)\mapsto (s,f(z))\in [0,\eps)\times \p M_2
    \end{equation}
    is smooth. As the map $f$ is a metric isometry it holds that for each $p \in M_1$, which is also in the domain of the $\varphi_1$, we have that
    \[
    s_1(p)=s_2(f(p)), \quad \text{ and } \quad 
    f(z_1(p))=z_2(f(p)).
    \]
    Therefore, the local representation $\varphi_2\circ f \circ \varphi_{1}^{-1}$ of the map $f$, is just given by the formula \eqref{eq:local_representation_of_f}. Thus, the map $f$ is also smooth near the boundary of $\p M_1$. By reversing the roles of $M_1$ and $M_2$ we can analogously show that $f^{-1}$ is smooth near $\p M_2$.

    Finally, to verify the pullback property $f^\ast F_2=F_1$.
     By the proof of the original version of the theorem 
    the pullback property holds over the interior. Because $f$ is a diffeomorphism of the whole manifold with boundary and the metrics are continuous, the pullback property holds over the boundary as well.
\end{proof}
We are ready to prove Corollary \ref{cor:stability_2}.

\begin{proof}[Proof of Corollary \ref{cor:stability_2}]
If $\hd{C(S_1)}(\CRR{X_2,S_2}(M_1),\cR_{X_2,S_2,\phi}(M_2))=0$ then by the same reasoning as above there is a metric isometry between the manifolds $M_1$ and $M_2$. Hence, the uniqueness claim of type (a) follows from the classical Myers--Steenrod result. The uniqueness claim of type (b) follows from the generalization of the Myers--Steenrod theorem, given in Proposition \ref{prop:MS_thm}. 
\end{proof}

As to be expect the proof of Theorem \ref{thm:stability_2} builds heavily on the proof of Theorem \ref{thm:stability_1}. However, in order to prove Theorem \ref{thm:stability_2} we need to justify the existence of the Lipschitz constant appearing in the inequality \eqref{eq:dist_of_data}. The following proposition achieves exactly this.

\begin{prop} \label{prop:ghd_equivalent}
The original and truncated Gromov--Hausdorff distances are comparable in the following sense:
    \begin{itemize}
    \item[(I)]
    Let $\eps>0$. If $(X,d_X)$ and $(Y,d_Y)$ are compact metrics spaces then
\begin{equation}
\label{eq:bi_Lip_estimate_1}
    \ghd{\eps}(X,Y) \leq  \ghd{}(X,Y).
\end{equation}
    \item[(II)]
    Let $\eps>0$ and $D>0$. If $(X,d_X)$ and $(Y,d_Y)$ are compact length spaces with 
    $\diam(X),\diam(Y)\leq D$ then
\begin{equation}
\label{eq:bi_Lip_estimate_2}
    \ghd{}(X,Y) \leq \Big(\frac{2D}{\eps} + 1\Big)\ghd{\eps}(X,Y) .
\end{equation}
    \end{itemize}
\end{prop}
\begin{proof}
Claim (I):
We first consider the case when $\ghd{}(X,Y) < \eps$. To verify the inequality in \eqref{eq:bi_Lip_estimate_1} we take any $j \in \N$ such that $\ghd{}(X,Y)+\frac{1}{j}<\eps$. Then by the definition of Gromov--Hausdorff distance there is a metric space $(Z,d_Z)$ and isometric embeddings $f\colon (X,d_X) \to (Z,d_Z)$ and $g\colon (Y,d_y)) \to (Z,d_Z)$ such that
    \[
    \hd{Z,d_Z}(f(X),g(Y))\leq \ghd{}(X,Y)+\frac{1}{j}<\eps.
    \]
    Thus it follows from Lemma \ref{lem:props_of_trunc} that
    \[
    \hd{Z,d_Z}(f(X),g(Y))=\hd{Z,d_{Z,\eps}}(f(X),g(Y)),
    \]
    and that the maps
    \[
    f \colon (X,d_{X,\eps}) \to (Z,d_{Z,\eps}), \quad \text{ and } \quad 
    g \colon (Y,d_{Y,\eps}) \to (Z,d_{Z,\eps})
    \]
    are metric isometries between the respective truncated spaces. Hence, the inequality \eqref{eq:bi_Lip_estimate_1} holds due to the arbitrarity of the integer $j$. 
    
    Since the diameter of $(X,d_{X,\eps})$ and $(Y,d_{Y,\eps})$ is at most $\eps$ we can use the estimate \eqref{prop:ghd_upper_bound}, in the case $\ghd{}(X,Y)\geq \eps$, to obtain the inequality \eqref{eq:bi_Lip_estimate_1}.

Claim (II):
    First consider the case where $\ghd{\eps}(X,Y) < \frac{\eps}{4}$. To verify the inequality \eqref{eq:bi_Lip_estimate_2} let $\delta \in (\ghd{\eps}(X,Y), \frac{\eps}{4})$. By \cite[Theorem 7.3.25]{burago2001course} we have
    \begin{align*}
        \ghd{\eps}(X,Y) := \frac{1}{2} \inf(\text{dis}_\eps(\mathfrak{R})),
    \end{align*}
    where the infimum is taken over all correspondences $\mathfrak{R}$ between the sets $X$ and $Y$. Furthermore, in the above we used the notation for the distortion
    \begin{equation}
    \label{eq:distortion_correspondence}
        \text{dis}_\eps(\mathfrak{R}) = \sup\{|d_{X,\eps}(x,x') - d_{Y,\eps}(y,y')|:(x,y),(x',y')\in\mathfrak{R}\},
     \end{equation}
    between the truncated spaces. 
    
    Since $\ghd{\eps}(X,Y) < \delta < \frac{\eps}{4}$, we may find a correspondence $\mathfrak{R}$ such that $\text{dis}_\eps(\mathfrak{R})\leq 2\delta < \frac{\eps}{2}$. Our goal is to show that the distortion of the correspondence $\mathfrak{R}$ between the original compact length spaces $(X,d_X)$ and $(Y,d_Y)$ is bounded from above by $2\delta \Big(\frac{2D}{\eps} + 1\Big)$. From here, we arrive in the estimate
    \begin{align*}
        \ghd{}(X,Y) \leq \delta \Big(\frac{2D}{\eps} + 1\Big),
    \end{align*}
    which holds true for any $\delta > 0$ satisfying $\ghd{\eps}(X,Y) < \delta < \frac{\eps}{4}$. Thus, we have verified the inequality \eqref{eq:bi_Lip_estimate_2}. 
    
    Let $x,x' \in X$ be such that $d_{X,\eps}(x,x') \leq  \frac{\eps}{2}$. Then we choose $y,y' \in Y$ such that $(x,y),(x',y')\in\mathfrak{R}$, and incorporate \eqref{eq:distortion_correspondence} to obtain the estimates 
    $$
        |d_{Y,\eps}(y,y') -d_{X,\eps}(x,x')| \leq \text{dis}_\eps(\mathfrak{R})\leq 2\delta < \frac{\eps}{2}
        \quad \text{and} \quad
        d_{Y,\eps}(y,y') < \eps.
    $$
    Therefore, 
    \[
    d_{X,\eps}(x,x') = d_X(x,x') \quad  \text{and} \quad  d_{Y,\eps}(y,y') = d_Y(y,y').
    \]
    In particular, we have proven that for any $(x,y),(x',y') \in \mathfrak{R}$ we have that
    $$
        |d_X(x,x') - d_Y(y,y')|\leq 2\delta,
    $$
    if $d_{X,\eps}(x,x')< \frac{\eps}{2}$. 

    Now let $(x_a,y_a), (x_b,y_b) \in \mathfrak{R}$ be arbitrary. Since $X$ is a compact length space, the points $x_a$ and $x_b$ can be connected by a distance minimizing curve. Thus, we can find a finite number of points $x_0, \ldots, x_{N}$ with $x_0=x_a$ and $x_N = x_b$ which satisfy the following three properties:
    \begin{align}
        &d_X(x_a,x_b) = \sum_{i=1}^N d_X(x_i,x_{i-1}), 
        &&d_X(x_i,x_{i-1}) = \frac{\eps}{2},
        &&&d_X(x_N,x_{N-1}) < \frac{\eps}{2} \label{eq:splitting}.  
    \end{align}
    By combining (\ref{eq:splitting}) we obtain the estimate $N \leq \frac{2}{\eps}d_X(x_a,x_b) + 1$. Since $\mathfrak{R}$ is a correspondence between $X$ and $Y$, we can find points $y_1,\ldots,y_{N-1}\in Y$ such that $(x_i,y_i)\in \mathfrak{R}$. Now triangle inequality,  (\ref{eq:splitting}), the previous part of the proof and the assumption $\diam(X)\leq D$ imply that 
    \begin{align*}
        d_Y(y_a,y_b) - d_X(x_a,x_b) 
        &\leq 2\delta\Big(\frac{2D}{\eps} + 1\Big).
    \end{align*}
    Since $Y$ is also a compact length space, whose diameter does not exceed $D$, we can repeat the steps above to show that
    $$
        |d_X(x_a,x_b) - d_Y(y_a,y_b)| \leq 2\delta \Big(\frac{2D}{\eps} + 1\Big), \quad \text{for all } (x_a,y_a), (x_b,y_b)\in\mathfrak{R}.
    $$
    Hence, the distortion of $\mathfrak{R}$ between $(X,d_X)$ and $(Y,d_Y)$ has the upper bound 
    $
         2\delta\Big(\frac{2D}{\eps} + 1\Big).   
    $
    For the other case, we assume that $\ghd{\eps}(X,Y)\geq \frac{\eps}{4}$. Since the diameters are bounded from above by the number $D$ we get from the estimate \eqref{prop:ghd_upper_bound} that 
    $$
        \ghd{}(X,Y)\leq \frac{D}{2} < \frac{D}{2} + \frac{\eps}{4} = \Big(\frac{2D}{\eps} + 1\Big)\Big(\frac{\eps}{4}\Big) \leq \Big(\frac{2D}{\eps} +1\Big)\ghd{\eps}(X,Y),
    $$
as claimed.
\end{proof}


\begin{rem}
\label{rem:truncation_and_isometry}
    We note that in the light of Remark \ref{rem:non_isometric_trunc} that part (II) of Proposition \ref{prop:ghd_equivalent} is not true without the length space assumption. We provide an example for the Lipschitz estimate of Part (II) of Proposition \ref{prop:ghd_equivalent}. Consider the closed intervals $[0,1]$ and $[0,2]$ and recall that their Gromov--Hausdorff distance is exactly $\frac{1}{2}$. For $\eps>0$ we consider the truncated intervals $[0,1]_\eps$ and $[0,2]_\eps$.
    If we choose $\eps = \frac{3}{4}$ and $D =2$, then
    \[
    \ghd{}([0,1]_\eps,[0,2]_\eps)\geq \frac{1}{8}.
    \]
    Since $\diam([0,1]), \diam([0,2])\leq D$ we have that $\frac{2D}{\eps} + 1 = \frac{19}{3}$, and
    $$
        \ghd{}([0,1],[0,2]) = \frac{1}{2} < \frac{19}{24} = \Big(\frac{19}{3}\Big)\Big(\frac{1}{8}\Big) \leq \Big(\frac{2D}{\eps} + 1\Big)\ghd{\eps}([0,1],[0,2]).
    $$
\end{rem}

We are ready to provide a proof for our first main Theorem \ref{thm:stability_2}.

\begin{proof}[Proof of Theorem~\ref{thm:stability_2}]
We first assume that 
\begin{equation}
    \label{eq:dist_of_data}
\hd{C(S_1)}(\CRR{X_1,S_1}(X_1),\CRR{X_2,S_2,\phi}(X_2))
 < \eps.
\end{equation}
By following the steps of the proof of Theorem \ref{thm:stability_1} we get from the assumption \eqref{eq:dist_of_data} that
\[
\ghd{\eps}((X_1,d_1),(X_2,d_2))
\leq 
\hd{C(S_1)}(\CRR{X_1,S_1}(X_1),\CRR{X_2,S_2,\phi}(X_2)).
\]

From here we use Part (II) of Proposition \ref{prop:ghd_equivalent} in conjunction with the assumed diameter bound to see that the inequality \eqref{eq:dist_of_manifolds} is true.
Then we suppose that the estimate \eqref{eq:dist_of_data} is false. Since the diameters of the truncated manifolds do not exceed the number $\eps$ we get from \eqref{prop:ghd_upper_bound} that 
$
\ghd{\eps}((X_1,d_1),(X_2,d_2))
\leq 
\frac{\eps}{2}.
$
Thus, the inequality \eqref{eq:dist_of_manifolds} follows trivially from Part (II) of Proposition \ref{prop:ghd_equivalent}.
The final uniqueness claim follows from the estimate \eqref{eq:dist_of_manifolds}  since compact metric spaces whose Gromov--Hausdorff distance is zero are isometric. 
\end{proof}

\section{Examples of \texorpdfstring{$\flie{\eps}$}{FLIEe}-manifolds}
\label{Sec:Appendix}

\subsection{Examples of \texorpdfstring{$\flie{\eps}$}{FLIEe} domains in \texorpdfstring{$\mathbb{R}^n$}{Rn}}

In this section we consider a smooth and bounded domain $\Omega \subset \R^n$ equipped with the natural length metric $d_\Omega$.  Note that if $\p \Omega$ has concave parts then $d_\Omega$ may not coincide with the Euclidean distance.
We provide three examples of $\flie{\eps}$-space $(\overline{\Omega},d_\Omega)$ with the closed measurement set $\partial\Omega$. As a basic example, due to Proposition \ref{prop:global_isometry}, all smooth and bounded convex subsets of $\R^n$ satisfy the $\flie{\eps}$-condition, for every $\eps>0$. 

\medskip

Next we show that the annulus $\Omega = B(0,R)\backslash \overline{B(0,r)} \subset \R^n$, for $0<r<R$, is $\flie{\eps}$ with $\eps = \frac{1}{2}\pi r$. 
Consider points $x_1,x_2\in \Omega$ such that $d_\Omega (x_1,x_2) < \frac{1}{2}\pi r$ and let $\sigma$ be a distance minimizing curve in $\overline{\Omega}$ connecting $x_1$ to $x_2$ which is not a straight line.
We can write $\sigma=\sigma_2 \circ \alpha \circ \sigma_1$ where $\sigma_i$ is a line segment from $x_i$ to $\mathbb{S}^{n-1}(r)$ and $\alpha$ is a part of a  great circle on $\mathbb{S}^{n-1}(r)$ such that $\alpha$ and $\sigma_i$ have the same velocity when they meet.
Thus, $d_\Omega(x_1,x_2) = s(x_1) + \ell + s(x_2)$ where $s(x_i)$ is the length of $\sigma_i$ and $\ell$ is the length of $\alpha$. 
Due to Pythagorean theorem $\sigma_i$ has length $s(x_i)=\sqrt{\|x_i\|^2-r^2}$. Furthermore, with respect to the round metric the distance of the intersection points $\mathbb{S}^{n-1}(r)\cap \sigma_i = \{w_i\}$ from the orthogonal projections $Px_i$ of $x_i$ onto $\mathbb{S}^{n-1}(r)$ is $r \delta(\|x_i\|))$, where $\delta(t) = \arccos(\frac{r}{t})$ for $t \geq r$.
That is,
$w_i \in C(x_i)
$
where these sets are boundaries of metric balls on the sphere $\mathbb{S}^{n-1}(r)$ centered at $Px_i$ having radius $r\delta(\|x_i\|)$ with respect to the round metric.

We extend $\sigma_2$ by a straight line to some point $z \in \mathbb{S}^{n-1}(r)$ and call this curve $\tilde \sigma_2$. We claim that the curve $\tilde \sigma=\tilde \sigma_2 \circ \alpha \circ \sigma_1$, whose length is  $s(x_1) + \ell + s(z)$,
is a distance minimizing curve from $x_1$ to $z$. 
To verify this, we recall that $\alpha$ is the section of $\sigma$ that is a segment of a great circle $\tilde \alpha$ say of $\mathbb{S}^{n-1}(r)$. 
Due to symmetries both points $Px_1$ and $Pz$ lie on the great circle $\tilde \alpha$. In fact, since $\sigma$ has length less than $\frac{1}{2}\pi r$, the distance $d_\Omega(Px_1,Pz)$ 
is strictly less than $\pi r$. Therefore, $\tilde\alpha$ is the only great circle containing both $Px_1$ and $Pz$, and $\alpha$ is a distance minimizing curve between the disjoint sets  
$C(x_1)$ and $C(z)$
which contain $w_1$ and $w_2$ respectively. 
Finally, we note that by following the same reasoning as above any distance minimizing curve from $x_1$ to $z$ needs to have the same structure as $\tilde \sigma$: starting as a line segment from $x_1$ to  $C(x_1)$ from where continuing as a segment of a great circle to $C(z)$, and ending as a line segment to $z$. By the previous argument $\tilde \sigma$ is a distance minimizing curve from $x_1$ to $z$.

\medskip

It was shown in \cite[Theorem 6.2]{wolter1985cut} that the condition (C) of Proposition \ref{prop:global_isometry} is satisfied for a planar domain if and only if the domain is simply connected. Thus, the closure of a planar domain $\Omega$ satisfies $\flie{\eps}$ for all $\eps > 0$ if and only if $\overline{\Omega}$ is simply connected. Hence, the global isometry property of the travel time map can be satisfied by some non-convex sets. 

\subsection{A spherically symmetric \texorpdfstring{$\flie{\eps}$}{AE}-manifold with a convex boundary and interior conjugate points}

This appendix provides an example of a Riemannian manifold  with strictly convex boundary and interior conjugate points,  which satisfies the $\flie{\eps}$-property for some $\eps > 0$. We provide numerical evidence for this claim by using Wolfram Mathematica.

\subsubsection{Herglotz manifolds and the \texorpdfstring{$\flie{\eps}$}{AE}-property}
We consider the closed Euclidean unit disk of $\R^2$ equipped with a radial Riemannian metric $g$ conformal to the Euclidean metric $e$. We write $g = c^{-2}e$ where $c$ is a positive smooth function called the sound speed. In Euclidean polar coordinates $(r,\theta)$ of the disk we write
\begin{equation}
g(r) = \frac{dr^2+r^2d\theta^2}{c(r)^2}.
\end{equation}


We say that the metric $g$ satisfies the \textit{Herglotz} condition if
\begin{equation}
\frac{d}{dr}
\left(
\frac{r}{c(r)}
\right)
>
0, \quad \text{ for all } r\in [0,1].
\end{equation}
If $M$ is the unit disk in $\R^2$ and $g$ is a spherically symmetric Riemannian metric on $M$ satisfying the Herglotz condition, we say that $(M,g)$ is a Herglotz manifold.

The spherically symmetric manifold $(M,g)$ is non-trapping if and only if the Herglotz condition is satisfied. This is also equivalent to the property: The level sets, circles in $\R^2$, of the sound speed $c$ are strictly convex with respect to the metric $g$. Thus, the boundary of a Herglotz manifold is strictly convex. Furthermore, if the Herglotz condition is satisfied then any geodesic of $(M,g)$ has a unique point closest to the origin, which we call the \textit{tipping point}. If the geodesic $\gamma$ is parameterized by the arc length then the tipping point of $\gamma$ is $\gamma(L(\gamma)/2)$, where $L(\gamma)$ stands for the length of $\gamma$. 

We divide every geodesic into two parts with respect to the tipping point. We call these parts the halves of the geodesics, and recall that in the polar coordinates on each half the radial component $r(t)$ of the geodesic $\gamma(t)$ is strictly monotonic. We have $\dot r < 0$ on one of the halves and $\dot r > 0$ on the other.

The following proposition gives a characterization of $\flie{\eps}$-property on Herglotz manifolds. Later we use this proposition to numerically verify that our example satisfies the $\flie{\eps}$-property.

\begin{prop}
\label{prop:herglotz-halfgeodesics}
Let $\eps > 0$. The following hold for a Herglotz manifold $(M,g)$.
\begin{enumerate}
    \item
    If $(M,g)$ is $\flie{\eps}$ then there is a uniform $\delta>0$ such that each geodesics $\gamma$ is minimizing on the segments $[0,\frac{L(\gamma)}{2}+\delta]$ and $[\frac{L(\gamma)}{2}-\delta,L(\gamma)]$. In particular, each geodesic is minimizing past the tipping point.
    \item
    If there is a uniform $\delta > 0$ so that any geodesic $\gamma$ with $\gamma(0) \in \partial M$ is minimizing from $\gamma(0)$ to $\gamma(\frac{L(\gamma)}{2} + \delta)$, then $(M,g)$ is $\flie{\eps}$ for all $0 < \eps < 2\delta$.
\end{enumerate}
\end{prop}

\begin{proof}
Claim 1:
    We recall that $i(M,g)$ stands for the injectivity radius of $(M,g)$, and  choose 
    $
    \delta=\min\left\{\frac{\eps}{3},\frac{i(M,g)}{3}\right\}.
    $
    Then we consider any unit speed geodesic $\gamma \colon [0,L(\gamma)] \to M$ so that $\gamma(\frac{L(\gamma)}{2})$ is the tipping point and $\gamma(0)$ and $\gamma(L(\gamma))$ are on the boundary.
    Then the points  $p=\gamma(\frac{L(\gamma)}{2}) - \delta)$ and $q = \gamma(\frac{L(\gamma)}{2}) + \delta)$ are closer than $\eps$ to each other and the curve $\gamma$ is the only minimizing curve from $p$ to $q$. Since $(M,g)$ is $\flie{\eps}$ at least one of the geodesic segments $\gamma|_{[0,\frac{L(\gamma)}{2})+\delta]}$ or $\gamma|_{[\frac{L(\gamma)}{2}-\delta,L]}$ is minimizing. Thus, by the spherical symmetry they both are. We have proved the claim in item (1).

Claim 2:
    We assume that each geodesic $\gamma$ with $\gamma(0) \in \partial M$ is minimizing from $\gamma(0)$ to 
    \\
    $\gamma(\frac{L(\gamma)}{2})+\delta)$, and choose $0 < \eps < 2\delta$. Suppose that $p,q \in M$ are so that $d(p,q) < \eps$. Since Herglotz-manifolds are non-trapping it follows from  Hopf-Rinow  theorem there is a unit speed geodesic $\gamma$ parametrized so that $\gamma(0) \in \partial M$, $\gamma(t) = p$ and $\gamma(s) = q$ for some $t < s$ and $\gamma|_{[t,s]}$ is minimizing. There are two cases on how the points $p$ and $q$ can lie on the geodesic $\gamma$. Either they are on the same half geodesic  or on the opposite sides of the tipping point $\gamma(\frac{L(\gamma)}{2})$.

    If $p$ and $q$ are on the same half geodesic, then by the radial symmetry we can assume that $s,t \in [0,\frac{L(\gamma)}{2}]$. Thus, by the assumption $\gamma$ is minimizing from $\gamma(0)$ to $q$. If $p$ and $q$ are on opposite sides of the tipping point $\gamma(\frac{L(\gamma)}{2})$ we have by the assumption 
    $
    d(p,q)<\eps<2\delta,
    $
    that $s,t \in [\frac{L(\gamma)}{2}-\delta,\frac{L(\gamma)}{2}+\delta]$. Therefore our assumption gives that $\gamma$ is minimizing from $\gamma(0)$ to $q$. 
\end{proof}

\subsubsection{Example Herglotz manifolds}
Consider the unit disk $M$ equipped with the Riemannian metric $g = c^{-2}e$ where the sound speed is defined by
\begin{equation}
\label{eqn:herglotz-example}
c(r)
=
\exp
\left(
-\frac k2
\exp
\left(
-\frac{r^2}{2\sigma^2}
\right)
\right)
\end{equation}
with $k=1.6$ and $\sigma=0.4$.
The behavior of this family of Riemannian metrics as a function of the two parameters was considered in~\cite[p.~16]{Monard2014}.
It is is easy to verify numerically that this metric satisfies the Herglotz condition, and Figure~\ref{fig:opening-angle} shows the existence of interior conjugate points.
Two intersecting geodesics of this metric are depicted in Figure~\ref{fig:opening-angle-test}.

We verify numerically with Wolfram Mathematica that it satisfies the $\flie{\eps}$-property for some $\eps > 0$.
The Mathematica code can be found in the accompanying file \texttt{conjugate-points-example.nb}.

\subsubsection{Interior conjugate points}
We follow the conventions and notations of~\cite{dHIK2022}.
For a given radius $r \in [0,1]$ let $\gamma(r)$ be a geodesic with end points on the unit circle and the tipping point at radius $r$. There are multiple choices for $\gamma(r)$ but for our purposes we can choose any. Recall that the length of either half of $\gamma(r)$ depends only on $r$ and can be computed by the formula
\begin{equation}
L(r)
=
\int_r^1
\frac{1}{c(s)}
\left(
1
-
\left(
\frac{rc(s)}{sc(r)}
\right)^2
\right)^{-1/2}
\,ds.
\end{equation}
The angular distance between the tipping point and either of the end points of $\gamma(r)$ also only depends on $r$ and can be computed by the formula
\begin{equation}
\label{eqn:opening-angle}
\alpha(r)
=
\int_r^1
\frac{rc(s)}{c(r)s^2}
\left(
1
-
\left(
\frac{rc(s)}{sc(r)}
\right)^2
\right)^{-1/2}
\,ds.
\end{equation}
We call $\alpha$ the opening angle. Note that $\alpha$ is $C^1$. It was shown in~\cite[Lemma 4.4]{dHIK2022} that the set of radii $r \in [0,1]$ so that the end points of $\gamma(r)$ are conjugate to each other along $\gamma(r)$ is
\begin{equation}
C
=
\{\,
r \in [0,1]
\,:\,
\alpha'(r) = 0
\,\}.
\end{equation}
This can be extended to give a method for confirming that the end points of $\gamma(r)$ on the level set $\{r = r_0\}$ are conjugate to each other along $\gamma(r)$. For a given radius $r_0 \in [0,1]$ define the function $\tilde{\alpha}(r;r_0)$ by replacing the upper limit of the integral in~\eqref{eqn:opening-angle} with $r_0$. Then $\tilde{\alpha}(r;r_0)$ computes the angular distance between the tipping point and either of the end points of $\gamma(r)$ on $\{r = r_0\}$. The proof of~\cite[Lemma 4.4]{dHIK2022} shows that the set of radii such that the end points of $\gamma(r)$ on $\{r=r_0\}$ are conjugate to each other along $\gamma(r)$ is
\begin{equation}
C_{r_0}
=
\{\,
r \in [0,r_0]
\,:\,
\partial_r\tilde{\alpha}(r;r_0) = 0
\,\}.
\end{equation}
A plot of the modified opening angle $\tilde{\alpha}$ for the metric $g$ is shown in Figure~\ref{fig:opening-angle}.

\begin{figure}
\caption{Two illustrations of our example of a Herglotz manifold.}
\label{fig:herglotz-wrapper}
\begin{subfigure}[t]{.45\textwidth}
    \centering
    \includegraphics[width=1.0\linewidth]{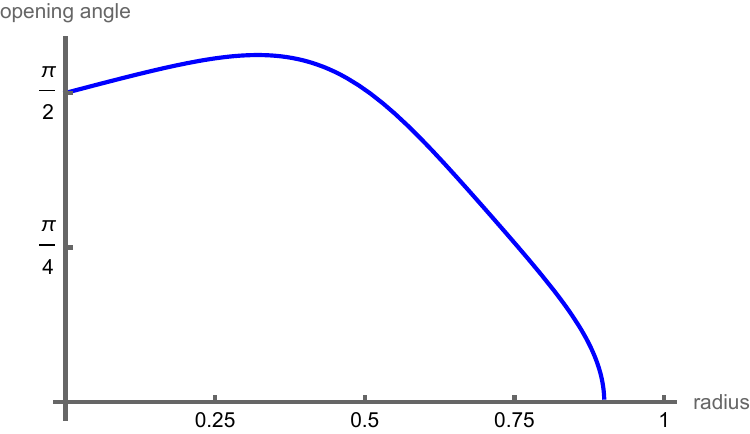}
    \caption{A numerical plot of the (half) opening angle $\tilde\alpha(r;r_0)$ as a function of $r$ for $r_0=.9$.
    There is clearly a critical point somewhere near $r=.32$, and the condition $\partial_r\tilde\alpha(r,r_0)=0$ corresponds exactly to a pair of conjugate points at radius $r_0$.
    Our model therefore has interior conjugate points.
    The opening angle at $r=0$ is always $\frac\pi2$.
    }
    \label{fig:opening-angle}
\end{subfigure}%
\hfill%
\begin{subfigure}[t]{.45\textwidth}
    \centering
    \includegraphics[width=0.8\linewidth,trim={0 .9cm 0 .9cm},clip]{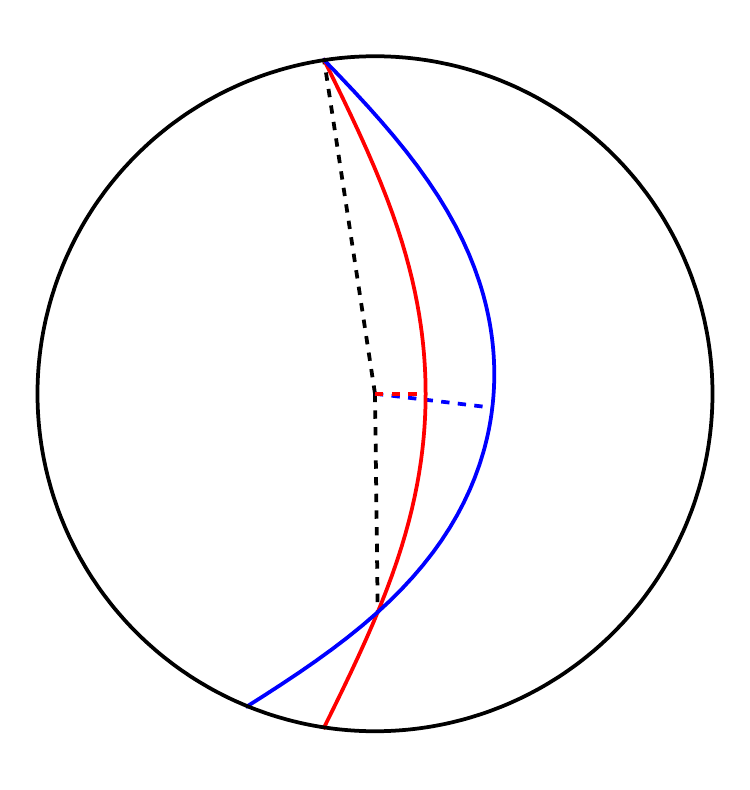}
    \caption{The \textcolor{red}{red} and \textcolor{blue}{blue} curves are geodesics of the sound speed of our example. The dashed lines indicate the location of the tipping points of the two geodesics. The dashed black lines highlight the intersection points of the two geodesics. See text for how such twice intersecting pairs of geodesics are detected numerically.
    }
    \label{fig:opening-angle-test}
\end{subfigure}
\end{figure}

\subsubsection{Verification of the  \texorpdfstring{$\flie{\eps}$}{AE}-property}
To verify that $(M,g)$ satisfies the $\flie{\eps}$-property we implement the following test numerically.

We begin with finding intersecting pairs of geodesics.
For a given radius $r_0 \in [0,1]$ find all radii $r_1 \in [0,1]$ so that the geodesics $\gamma(r_0)$ and $\gamma(r_1)$ have a common end point and intersect an additional time at radius $r_2 \in [\max(r_0,r_1),1)$ away from the origin. Let $S$ be the set of such radii triplets.

    We use the following strategy to find the set $S$ numerically. Define the another variant of the opening angle function by
    \begin{equation}
    \alpha(r;r',r'')
    =
    \int_{r'}^{r''}
    \frac{rc(s)}{c(r)s^2}
    \left(
    1
    -
    \left(
    \frac{rc(s)}{sc(r)}
    \right)^2
    \right)^{-1/2}
    \,ds
    \end{equation}
    for radii $0 < r \le r' \le r'' \le 1$. That is $\alpha(r;r',r'')$ computes the angular distance along the geodesic $\gamma(r)$ between its point at radius $r'$ and its point at radius $r''$. Then we see that $(r_0,r_1,r_2) \in S$ if and only if
    \begin{equation}
    \label{eqn:angle-test1}
    \alpha(r_0;r_0,r_2)
    +
    \alpha(r_0;r_0,1)
    =
    2\pi m\pm
    \alpha(r_1;r_2,1)
    \end{equation}
    or
    \begin{equation}
    \label{eqn:angle-test2}
    \alpha(r_0;r_0,r_2)
    +
    \alpha(r_0;r_0,1)
    =
    2\pi m\pm
    [\alpha(r_1;r_1,r_2)
    +
    \alpha(r_1;r_1,1)
    ]
    \end{equation}
    for some $m\in\Z$.
    The integer $m$ describes the relative winding number of the two geodesic segments.
    Conditions~\eqref{eqn:angle-test1} and~\eqref{eqn:angle-test2} can be tested explicitly. The strategy used to find $S$ is illustrated in Figure~\ref{fig:opening-angle-test}.
    

For all triplets of radii in $S$ we compute the lengths of the segments of the two geodesics between the shared points similarly to the angle computation described above.
We define a subset $\tilde S\subset S$ keeping only the cases ones where minimization fails, and for $\tilde S$ we compute the amount of time each geodesic extends beyond the midpoint.
Numerical computations suggest that these times are bounded from below by $0.29$.
Therefore in Proposition~\ref{prop:herglotz-halfgeodesics} we can set $\delta=0.29$ and conclude that our example manifold is $\flie{0.29}$.

\subsection{A closed \texorpdfstring{$\flie{\eps}$}{AE}-manifold with pairs of conjugate points outside the measurement set}

    We consider a closed Riemannian manifold with a measurement set $S$ which is a closure of an open set. In this example we are allowed to measure distances through $S$.
    Let $\Gamma \subset \mathbb{S}^2$ be a half of a great circle connecting the north and the south poles. We observe that this is the smallest set of $\mathbb{S}^2$ which intersects with all great circles. Thus, in the light of Lemma \ref{lem:lenght_of_geos} our measurement set should include this type of set. However, as explained in the introduction of this paper the travel time map $\CRR{\mathbb{S}^2,\Gamma}$  is not one-to-one. Thus, we study a closure of a neighborhood of the set $\Gamma$.  
    
    Let $r>0$ be small and define
    \[
    S=\overline{\bigcup_{x \in \Gamma} B(x,r)}.
    \]
     Clearly the set $\mathbb{S}^2 \setminus S$ has many circular arcs  which are longer than $\pi$. Thus, the set $\mathbb{S}^2 \setminus S$ has conjugate pairs and not all geodesics contained in this set are globally distance minimizing.  
    We claim that the sphere $\mathbb{S}^2$ with the closed measurement set $S$ satisfies the $\flie{\eps}$-property for  $\eps=r$, but not for all $\eps>0$.

    First we note that each great circle meets the set $S$ and this can happen in two different ways.
    Let $C$ be a great circle and denote the polar caps $B(\pm e_3,r)$ by $P_1$ and $P_2$ respectively. If $C$ does not meet $P_1\cup P_2$ then $C \cap S$ has only one connected component whose length is at least $2r$. If $C$ meets $P_1$ then by symmetry it also meets $P_2$. Thus, by the definition of the set $S$ it holds that $C \cap S$ has two connected components and at least one of them has a length greater or equal to $2r$.  

    Let $p,q \in \mathbb{S}^2$ be two points that are not antipodal. If $p \in S$ or $q \in S$ we have nothing to prove. Thus we may assume that $p,q \in \mathbb{S}^2 \setminus S$   
    and let $C$ be the unique great circle connecting these two points. 
    Suppose that,  $d(p,q)< 2r$, and let $A$ be a connected component of $S \cap C$ which has a length larger or equal to $2r$. We split $C$ into a short and long circular arcs $C_1$ and $C_2$ which both start from $p$ and end at $q$. Since $d(p,q)<2r$ and $p,q \notin S$ the set $A$ must be contained in $C_2$. 
    
    Without loss of generality we assume that $q$ is closer to $A$ than $p$ and let $z \in A$ be the closest point of $A$ to $q$.  Since $d(p,q)<2r$ it holds that the set $-A:=\{-x: \: x \in A\}$ cannot be completely contained in $C_1$. Thus, $p$, $q$ and $z$ must all be in the same half circle of $C$ which is then the shortest path of $\mathbb{S}^2$ from $p$ to $z$. We have proved that 
   $ d(p,z)=d(p,q)+d(q,z).
   $
    This equality implies that $\mathbb{S}^2$ and closed measurement set $S$ satisfy the $\flie{\eps}$-property for $\eps=r$. 

    Finally, we would like to emphasize that since we are allowed to measure the distances through the set $S$ it does not matter if the set $C_1 \cap (C\cap S)$ is not empty. However, the choice $d(p,q)<2r$ is crucial since if the $d(p,q) \geq 2r$ the set $C_1$ can contain $A$ or $-A$ and it would be always faster to go from $p$ to $z$ along $C_2$. In this case 
    the $\flie{\eps}$ property fails. Thus, $\mathbb{S}^2$ with the closed measurement set $S$ is not $\flie{\eps}$ for all $\eps>0$.
    
\bibliographystyle{abbrv}
\bibliography{bibliography}

\def\cprime{$'$}
\begin{thebibliography}{10}

\bibitem{alexander1981geodesics}
R.~Alexander and S.~Alexander.
\newblock Geodesics in {R}iemannian manifolds-with-boundary.
\newblock {\em Indiana University Mathematics Journal}, 30(4):481--488, 1981.

\bibitem{belishev1992reconstruction}
M.~I. Belishev and Y.~V. Kurylev.
\newblock To the reconstruction of a {R}iemannian manifold via its spectral data ({B}c--{M}ethod).
\newblock {\em Communications in partial differential equations}, 17(5--6):767--804, 1992.

\bibitem{bosi2022reconstruction}
R.~Bosi, Y.~Kurylev, and M.~Lassas.
\newblock Reconstruction and stability in {G}elfand’s inverse interior spectral problem.
\newblock {\em Analysis \& PDE}, 15(2):273--326, 2022.

\bibitem{burago2001course}
D.~Burago, Y.~Burago, and S.~Ivanov.
\newblock {\em A course in metric geometry}, volume~33.
\newblock American Mathematical Society Providence, RI, 2001.

\bibitem{burago2020quantitative}
D.~Burago, S.~Ivanov, M.~Lassas, and J.~Lu.
\newblock Quantitative stability of {G}el'fand's inverse boundary problem.
\newblock {\em arXiv preprint arXiv:2012.04435}, 2020.

\bibitem{dHIK2022}
M.~V. de~Hoop, J.~Ilmavirta, and V.~Katsnelson.
\newblock Spectral rigidity for spherically symmetric manifolds with boundary.
\newblock {\em J. Math. Pures Appl. (9)}, 160:54--98, 2022.

\bibitem{de2024geometric}
M.~V. de~Hoop, J.~Ilmavirta, A.~Kykk{\"a}nen, and R.~Mazzeo.
\newblock Geometric inverse problems on gas giants.
\newblock {\em arXiv preprint arXiv:2403.05475}, 2024.

\bibitem{de2020foliated}
M.~V. de~Hoop, J.~Ilmavirta, M.~Lassas, and T.~Saksala.
\newblock A foliated and reversible {F}insler manifold is determined by its broken scattering relation.
\newblock {\em Pure and Applied Analysis}, 3(4):789--811, 2022.

\bibitem{dehoop2020determination}
M.~V. de~Hoop, J.~Ilmavirta, M.~Lassas, and T.~Saksala.
\newblock Determination of a compact {F}insler manifold from its boundary distance map and an inverse problem in elasticity.
\newblock {\em {C}ommunications in {A}nalysis and {G}eometry}, 31(7):1693--1747, 2023.

\bibitem{de2021determination}
M.~V. de~Hoop, J.~Ilmavirta, M.~Lassas, and T.~Saksala.
\newblock Determination of a compact {F}insler manifold from its boundary distance map and an inverse problem in elasticity.
\newblock {\em Communications in Analysis and Geometry}, 31(7):1293--1747, 2023.

\bibitem{de2021stable}
M.~V. de~Hoop, J.~Ilmavirta, M.~Lassas, and T.~Saksala.
\newblock Stable reconstruction of simple {R}iemannian manifolds from unknown interior sources.
\newblock {\em Inverse Problems}, 39(9):095002, 2023.

\bibitem{de2018inverse}
M.~V. de~Hoop and T.~Saksala.
\newblock Inverse problem of travel time difference functions on a compact {R}iemannian manifold with boundary.
\newblock {\em The Journal of Geometric Analysis}, 29(4):3308--3327, 2019.

\bibitem{halverson2008bing}
D.~M. Halverson and D.~Repov{\v{s}}.
\newblock The {B}ing-{B}orsuk and the {B}usemann conjectures.
\newblock {\em Mathematical Communications}, 13(2):163--184, 2008.

\bibitem{helin2016correlation}
T.~Helin, M.~Lassas, L.~Oksanen, and T.~Saksala.
\newblock Correlation based passive imaging with a white noise source.
\newblock {\em Journal de Math{\'e}matiques Pures et Appliqu{\'e}es}, 116:132--160, 2018.

\bibitem{ilmavirta2023three}
J.~Ilmavirta, B.~Liu, and T.~Saksala.
\newblock Three travel time inverse problems on simple {R}iemannian manifolds.
\newblock {\em Proceedings of the American Mathematical Society}, 151(10):4513--4525, 2023.

\bibitem{ivanov2018distance}
S.~Ivanov.
\newblock Distance difference representations of {R}iemannian manifolds.
\newblock {\em Geometriae Dedicata}, 207(1):167--192, 2020.

\bibitem{ivanov2020distance}
S.~Ivanov.
\newblock Distance difference functions on nonconvex boundaries of {R}iemannian manifolds.
\newblock {\em St. Petersburg Mathematical Journal}, 33(1):57--64, 2022.

\bibitem{Katchalov2001}
A.~Katchalov, Y.~Kurylev, and M.~Lassas.
\newblock {\em Inverse boundary spectral problems}, volume 123 of {\em Monographs and Surveys in Pure and Applied Mathematics}.
\newblock Chapman \& Hall/CRC, Boca Raton, FL, 2001.

\bibitem{katsuda2007stability}
A.~Katsuda, Y.~Kurylev, and M.~Lassas.
\newblock Stability of boundary distance representation and reconstruction of {R}iemannian manifolds.
\newblock {\em Inverse Problems \& Imaging}, 1(1):135, 2007.

\bibitem{kurylev1997multidimensional}
Y.~Kurylev.
\newblock Multidimensional {G}e{l'}fand inverse problem and boundary distance map.
\newblock {\em Inverse Problems Related with Geometry (ed. H. Soga)}, pages 1--15, 1997.

\bibitem{kurylev2010rigidity}
Y.~Kurylev, M.~Lassas, and G.~Uhlmann.
\newblock Rigidity of broken geodesic flow and inverse problems.
\newblock {\em American journal of mathematics}, 132(2):529--562, 2010.

\bibitem{lassas2015determination}
M.~Lassas and T.~Saksala.
\newblock Determination of a {R}iemannian manifold from the distance difference functions.
\newblock {\em Asian journal of mathematics}, 23(2):173--200, 2019.

\bibitem{lee2018introduction}
J.~M. Lee.
\newblock {\em Introduction to Riemannian manifolds}, volume~2.
\newblock Springer, 2018.

\bibitem{matveev2017myers}
V.~Matveev and M.~Troyanov.
\newblock The {M}yers-{S}teenrod theorem for {F}insler manifolds of low regularity.
\newblock {\em Proceedings of the American Mathematical Society}, 145(6):2699--2712, 2017.

\bibitem{meyerson2020stitching}
R.~Meyerson.
\newblock Stitching data: Recovering a manifold’s geometry from geodesic intersections.
\newblock {\em The Journal of Geometric Analysis}, 32(3):95, 2022.

\bibitem{Monard2014}
F.~Monard.
\newblock Numerical implementation of geodesic {X}-ray transforms and their inversion.
\newblock {\em SIAM J. Imaging Sci.}, 7(2):1335--1357, 2014.

\bibitem{myers-steenrod}
S.~B. Myers and N.~E. Steenrod.
\newblock The group of isometries of a {R}iemannian manifold.
\newblock {\em Annals of Mathematics}, 40(2):400--416, 1939.

\bibitem{o1983semi}
B.~O’Neill.
\newblock {\em Semi-Riemannian geometry with applications to relativity}.
\newblock Pure and Applied Mathematics/Academic Press, Inc, 1983.

\bibitem{paternain2023geometric}
G.~P. Paternain, M.~Salo, and G.~Uhlmann.
\newblock {\em Geometric inverse problems}, volume 204.
\newblock Cambridge University Press, 2023.

\bibitem{pavlechko2022uniqueness}
E.~Pavlechko and T.~Saksala.
\newblock Uniqueness of the partial travel time representation of a compact {R}iemannian manifold with strictly convex boundary.
\newblock {\em Inverse Problems \& Imaging}, 16(5):1325--1357, 2022.

\bibitem{petersen2006riemannian}
P.~Petersen.
\newblock {\em Riemannian geometry}, volume 171.
\newblock Springer, $3^{rd}$ edition, 2006.

\bibitem{Vandaele-graph}
R.~Vandaele, B.~Rieck, Y.~Saeys, and T.~{De Bie}.
\newblock Stable topological signatures for metric trees through graph approximations.
\newblock {\em Pattern Recognition Letters}, 147:85--92, 2021.

\bibitem{wolter1985cut}
F.-E. Wolter.
\newblock {\em Cut loci in bordered and unbordered Riemannian manifolds}.
\newblock PhD thesis, Technische Universit{\"a}t Berlin, 1985.

\end{thebibliography}

\end{document}